\documentclass[a4paper,oneside,11pt]{article}
 
\usepackage[a4paper]{geometry}
\usepackage{aeguill}                 
\usepackage{graphicx,calc}
\usepackage{amsmath,amsfonts,amssymb,amsthm}
\usepackage{pstricks} 
\usepackage{epstopdf}
\usepackage{srcltx}
\usepackage{comment}
\usepackage{hyperref}
\usepackage[all]{xy}
\usepackage[utf8]{inputenc} 
\usepackage[T1]{fontenc} 

\newtheorem{thmA}{Theorem}

\newtheorem{corA}[thmA]{Corollary}

\newtheorem{theorem}{Theorem}[section]

\newtheorem{proposition}[theorem]{Proposition}
\newtheorem{lemma}[theorem]{Lemma}
\newtheorem{corollary}[theorem] {Corollary}
\newtheorem{addendum}[theorem] {Addendum}

\theoremstyle{remark}
\newtheorem{remark}[theorem]{Remark}
\newtheorem{remarks}[theorem]{Remarks}

\theoremstyle{definition}

\newcommand{\und}[1]{{\underline{#1}}}
\newcommand{\ov}[1]{{\overline{#1}}}

\def\L{\Lambda}
\def\co{\colon}

\def\aut{{\rm{Aut}}}
\def\Aut{{\rm{Aut}}}
\def\out{{\rm{Out}}}
\def\Out{{\rm{Out}}}
\def\G{\Gamma}

\def\stab{{\rm{Stab}}}
\def\Stab{{\rm Stab}}
\def\fix{{\rm{Fix}}}
\def\<{\langle}
\def\>{\rangle}
\def\FS{{\mathcal{FS}}}
\def\Z{\mathbb{Z}}

\def\2rose{{ {$2$--rose}}}
\def\sym{{\rm{Sym}}}  

\newcommand{\LL}{A}

\newcommand{\rk}{{\rm{rk}}}
 
\newcommand{\Inn}{\rm{Inn}}
\newcommand{\ad}{{\rm{ad}}}

\newcommand{\Isom}{\rm{Isom}}
\newcommand{\Fix}{{\rm{Fix}}}

\newcommand{\ia}{\rm{\rm IA}_N(\mathbb{Z}/3\mathbb{Z})}
\newcommand{\iat}{\rm{\rm IA}_3(\mathbb{Z}/3\mathbb{Z})}
\newcommand{\pr}{{{\mathrm{rk}_F}}}
\newcommand{\gl}{\mathrm{GL}}

\newcommand{\image}{{\rm{Im}}}
\newcommand{\st}{\widetilde{\rm{Stab}}(T)}
\newcommand{\sto}{\widetilde{\rm{Stab}}^0(T)}

\title{Direct products of free groups in $\Aut(F_N)$}
\author{Martin R. Bridson and Richard D. Wade}

\begin{document}

\maketitle

\begin{abstract} We give a complete description of the embeddings of direct products of nonabelian free groups into $\aut(F_N)$ and $\out(F_N)$ when the number of direct factors is maximal. To achieve this, we prove that the image of
each such embedding has a canonical fixed point of a particular type in the boundary of Outer space.
\end{abstract}

\section{Introduction}

Mapping class groups of surfaces of finite type
and automorphism groups of free groups are central objects in geometric topology and group theory, and it is natural to study them in parallel.
Nielsen-Thurston theory \cite{wpt} provides a potent geometric description of the individual elements of mapping class groups,
and the train-track technology initiated by Bestvina and Handel \cite{BH} provides an equally potent description in the wilder
setting of free group automorphisms. In pursuit of a more global understanding of these groups, one seeks insight from their
actions on Teichm\"uller space and Outer space, as well as associated spaces such as the curve complex, in the case of
mapping class groups, and free factor and splitting complexes in the case of $\aut(F_N)$ and $\out(F_N)$. 

The need for geometric insights and invariants comes into particularly sharp focus when 
one is trying to elucidate the intricate subgroup structure of these groups, as we are in this article. One sees
this clearly in the classification of abelian subgroups, which has  many ramifications. From the Nielsen-Thurston 
theory, one knows that, up to finite-index, every abelian subgroup of the mapping class group 
is generated by combinations of Dehn twists in a collection of disjoint curves and, optionally,
a pseudo-Anosov automorphism on each connected component of the complement of these curves \cite{MR726319, MR1195787}.
 This description provides a 
host of geometric invariants for studying the totality of abelian subgroups, starting with the stable laminations of the pseudo-Anosov pieces and compatibility conditions for the curve systems. With these in hand, one can organise the commensurability
classes of maximal-rank abelian subgroups into a space that is closely related to the curve complex.  This idea
is central to  Ivanov's proof \cite{Iva} of commensurator rigidity for mapping class groups: he proved that,
with some low-genus exceptions, every isomorphism between finite-index subgroups of a mapping class group
is the restriction of a conjugation in the ambient group. 
Although the situation in free groups is much more complicated, Feighn and Handel \cite{FeH} succeeded in
describing all abelian subgroups of $\Out(F_N)$, and this description was used
by Farb and Handel \cite{FH} to establish commensurator rigidity for $\out(F_N)$ in the case $N \geq 4$. 

Ivanov's commensurator rigidity theorem was later extended by
Bridson, Pettet and Souto \cite{BPS} to various subgroups
of the mapping class group; they followed a similar template of proof  but used
 {\em direct products of nonabelian free groups} in place of abelian subgroups, replacing the curve complex with a complex built from decompositions of the surface into subsurfaces of euler characteristic $-2$ (cf. \cite{BM}).
In the same spirit, by focussing on direct products of nonabelian free groups rather than
abelian subgroups, Horbez and Wade \cite{HW2} proved that $\out(F_N)$ has commensurator rigidity for $N\ge 3$,
as do many of its natural subgroups.

Our main purpose in this article is to give a complete classification of the maximal-rank direct products of free groups 
in $\aut(F_N)$ and $\out(F_N)$; we shall see that they are remarkably rigid. In a companion to this paper \cite{BW2}, we shall
use this classification, in harness with \cite{BB},
to prove that $\aut(F_N)$ and its Torelli subgroup are commensurator rigid if $N\ge 3$.

The most important step in our proof of the
classification is a fixed-point theorem that we establish for the action of $\out(F_N)$ on the space
of free splittings of $F_N$ (Theorem \ref{t:thmA}). To motivate this theorem, we begin by describing an example of 
a subgroup of $\out(F_N)$ that is a direct product of the maximal number of copies of $F_2$; 
Horbez and Wade  \cite{HW2} proved that this number is $2N-4$ (one less than the cohomological dimension). 

We fix a basis $\{a_1, a_2, x_1, \ldots, x_{N-2} \}$ of $F_N$ and consider  the direct product $D$ of 
the $2N-4$ copies of $F_2$ in $\out(F_N)$ obtained by multiplying the elements $x_1, \ldots, x_{N-2}$ on the left and right by elements of $\langle a_1, a_2 \rangle$. This group $D$ fixes a graph-of-groups decomposition of $F_N$ with a single vertex group given by $A=\langle a_1, a_2\rangle$ and $N-2$ loops with trivial edge stabilizers (with $x_1, \ldots, x_{N-2}$ as the stable letters). The Bass--Serre tree associated to any such decomposition lies in boundary of Culler and Vogtmann's Outer space \cite{CV}; we call such a point a \emph{collapsed rose with $N-2$ petals}, as it is obtained from a rose in the interior of Outer space by collapsing two petals.

\begin{figure}[ht]  \centering \def\svgwidth{170pt} 
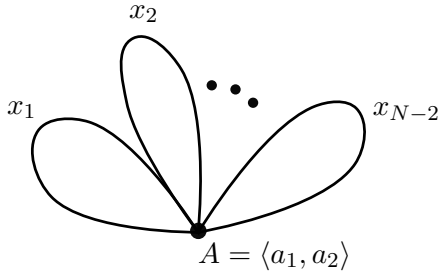 \caption{A collapsed rose with $N-2$ petals.} \label{f:rose} \end{figure}

\noindent Our first theorem shows that if a direct product of  nonabelian free groups in $\out(F_N)$ has the
maximal number of factors, then it has a canonical fixed point of this type. Note that in the following theorem, and throughout, we do not assume that $D$ is finitely generated. 

\begin{thmA}\label{t:thmA}
Let $N \geq 3$ and suppose $D<\Out(F_N)$ is a direct product of $2N-4$ nonabelian free groups. Then, in the boundary of Outer space, there is a unique collapsed rose with $N-2$ petals that is fixed by $D$.
\end{thmA}

When a subgroup of $\Out(F_N)$ fixes a tree $T$ in Outer space or its boundary, the preimage of this group in $\aut(F_N)$ admits an action on $T$. With a small amount of extra work, the following result can be deduced from Theorem~\ref{t:thmA}.

\begin{thmA}\label{t:B}
Let $N \geq 3$ and suppose $D<\aut(F_N)$ is a direct product of $2N-3$ nonabelian free groups. Then,
the image of $D$ in $\out(F_N)$ fixes a unique collapsed rose with $N-2$ petals, and $D$ acts on the Bass--Serre tree of this collapsed rose with a unique global fixed point. 
\end{thmA}

In order to move from
Theorems~A and B to the precise algebraic description of the direct products of free groups that we seek, some more
notation is required. We continue to work with a fixed basis $\{a_1, a_2, x_1, \ldots, x_{N-2}\}$ of $F_N$ and let $A$ be the free factor generated by $a_1$ and $a_2$. Let $L_i$ be the free group of rank 2 in $\aut(F_N)$ consisting of elements that send $x_i \mapsto ax_i$ for some $a \in A$ and fix all other basis elements. Similarly, we use $R_i$ to denote the free group of right transvections of $x_i$ by an element of $A$. Furthermore, we let $I(A)$ be the group of inner automorphisms generated by elements of $A$ and let $\tau$ be the Nielsen automorphism mapping $a_1 \mapsto a_1a_2$ and fixing all other basis elements. We let $L_i^\tau$, $R_i^\tau$, and $I(A)^\tau$ be the respective subgroups of these groups that commute with $\tau$ (equivalently, the elements from $A$ used in their associated transvections or inner automorphisms belong to $\Fix(\tau) \cap A = \langle a_1a_2a_1^{-1}, a_2 \rangle$).

\begin{thmA}\label{t:direct-products-aut}
Let $N \geq 3$ and suppose $D < \aut(F_N) $ is a direct product of $2N-3$ nonabelian free groups. Then a conjugate of $D$ is contained in one of the following groups.
\begin{itemize}
\item $L_1 \times \cdots \times L_{N-2} \times R_1 \times \cdots \times R_{N-2} \times I(A)$
\item $L_1^\tau \times \cdots \times L_{N-2}^\tau \times R_1^\tau \times \cdots \times R_{N-2}^\tau \times I(A)^\tau \times \langle \tau \rangle $
\item $\langle \tau, L_1 \rangle \times L_2^\tau \times \cdots L_{N-2}^\tau \times R_1^\tau \times \cdots R_{N-2}^\tau \times I(A)^\tau$
\item $L_1^\tau \times \cdots \times L_{N-2}^\tau \times R_1^\tau \times \cdots \times R_{N-2}^\tau \times \langle I(A), \tau \rangle $
\end{itemize}
\end{thmA}

In Section 7 we shall prove that  every direct product of $2N-4$ nonabelian free groups in $\out(F_N)$ is the image of 
one of the subgroups listed in Theorem~\ref{t:direct-products-aut}.

It is easy to see that if one of the subgroups listed in Theorem~\ref{t:direct-products-aut}
does not have a factor contained in a copy of $M_\tau=F_2\rtimes_\tau\Z$, then that subgroup is contained in a maximal
subgroup of the same form. In Section 8, we shall prove the less obvious fact that every nonabelian free group
of $M_\tau$ is also contained in a maximal one.

\begin{corA}
Let $N \geq3$ and let $\mathcal{D}$ be the family of subgroups of either $\aut(F_N)$ or $\out(F_N)$ that are direct products of $2N-3$ or $2N-4$ nonabelian free groups, respectively. Then every $D \in \mathcal{D}$ is contained in a maximal element (with respect to inclusion).
\end{corA}

A further consequence of Theorem~\ref{t:direct-products-aut} is that the centralizer of the direct product
$D$ is cyclic, and when it is non-trivial it is generated by a Nielsen automorphism. This yields the following rigidity result,
which plays a crucial role in \cite{BW2}.

\begin{thmA}\label{t:nielsen}
Let $\G$ be a finite-index subgroup of $\Aut(F_N)$ with $N \geq 3$  and let $f \colon \G \to \Aut(F_N)$ be an injective map. Every 
power of a Nielsen automorphism is mapped to a power of a Nielsen automorphism under $f$.
\end{thmA}

\subsection*{Techniques and proofs}

In the remainder of this introduction, we shall describe the structure of this paper and sketch some of the main ideas that
go into the proofs of the main results. Throughout, we assume $N \geq 3$. 

Our first goal is to establish the existence of the fixed point described in Theorem \ref{t:thmA}. 
Let $D$ be a direct product of $2N-4$ nonabelian free groups in $\out(F_N)$ and let $\widetilde D$ be 
its preimage in $\aut(F_N)$.  Our starting point
is Theorem~6.1 from \cite{HW2}, where actions of $\out(F_N)$ on relative free factor complexes 
were used to show that $D$ fixes a one-edge nonseparating free splitting of $F_N$. 
Lemma \ref{l:lifting_condition} tells us that a
 simplicial tree  $T$ in the boundary of Outer space will be fixed by $D$  if and only if the $F_N$-action on $T$ extends to an action of $\widetilde D $ on $T$ (where $F_N$ is identified with the group of inner automorphisms in $\widetilde D$). 
We apply this to the one-edge splitting fixed by $D$,  blowing-up the
action of $\widetilde D$ on the Bass--Serre tree to obtain the action on a
 collapsed rose with $N-2$ petals that we seek; this blow-up, which is described in Section~\ref{s:main_thm},
  is constructed using a  \emph{graph of actions} in the
 sense of Levitt \cite{Lev2}. A key point is to argue that  the  stabilizer in $\widetilde D$ of a vertex $v\in T$
  acts on a collapsed rose with $N-3$ petals, and that the adjacent edge stabilizers $\widetilde D_e$ for the one-edge splitting are elliptic in this new action; this last property implies the edges of the old tree  can be glued onto the
  new tree in a coherent fashion.

 The uniqueness of the fixed point described in Theorem~A is tackled separately. By work of Guirardel and Horbez \cite[Section~6]{GHmeasure}, if there were two collapsed roses with $N-2$ petals fixed by $D$ then they would belong to the same deformation space, and a folding path between these two collapsed roses would have to be fixed by $D$. However, $D$ is too large to fix any graph of groups decomposition of $F_N$ with more than one vertex group, which means the folding path must be trivial. Details are given in Section~\ref{s:fix}. 

Section \ref{s:4} contains an analysis of the stabilizers in $\aut(F_N)$ of collapsed roses and similar graphs. This
analysis plays a significant role in the proof of Theorem \ref{t:thmA}, and it renders the deduction of  Theorem \ref{t:B}
straightforward, as we shall see at the end of Section~\ref{s:main_thm}. The analysis of stabilizers of collapsed
roses also provides a crucial bridge from Theorem \ref{t:thmA} to Theorem
\ref{t:direct-products-aut}. In particular, with Theorem \ref{t:thmA} in hand, 
Proposition \ref{p:stab-description-rose} essentially reduces Theorem \ref{t:direct-products-aut} to an analysis
of the ways in which a direct product of $k+1$ nonabelian free groups can embed in 
\[ M_k(F_2) := F_2^k \rtimes \aut(F_2), \]
where the action of $\aut(F_2)$ in this semidirect product is diagonal. These embeddings are described in
Theorem \ref{t:all-the-D}; the required algebra is surprisingly delicate.
Given that our main results involve free groups of higher rank, it seems incongruous that special features of
$\aut(F_2)$ should play a crucial role at this stage of the proof, but nevertheless this is the case.
A key fact that makes many arguments in Section~\ref{s:slim_insert} work is that powers of Nielsen transformations are the only
 automorphisms of $F_2$ that have nonabelian fixed subgroups \cite{CT}.  
This special property lies behind the appearance of the Nielsen transformation $\tau$ in Theorem \ref{t:direct-products-aut}.  

\subsection*{Acknowledgements}  We are grateful to Mladen Bestvina, Sebastian Hensel and Camille Horbez  
for many stimulating conversations related to this work. 
The second author is supported by a University Research Fellowship from the Royal Society.

\section{Product rank, splittings, and automorphic lifts}

\subsection{Direct products of free groups}

A group $G$ has \emph{product rank} $\pr(G)=k$ if $k$ is the largest integer  such that $G$ contains a direct product of $k$ nonabelian free groups  (possibly $\pr(G)=\infty$). To understand how product rank behaves with respect to homomorphisms between groups, we make use of the following standard lemma.

\begin{lemma}\label{l:free_subgroup}
Let $K$ be a normal subgroup of a direct product of nonabelian free groups $G_1 \times G_2 \times \cdots \times G_k$. Suppose that $\pr(K)=l$. Then, after reordering the factors, $K$ is a normal subgroup of $G_1 \times G_2 \times \cdots \times G_l\times 1\times\dots\times 1$.
\end{lemma}

\begin{proof}
Let $g=(g_1,g_2, \ldots,g_k) \in K$ and without loss of generality assume $g_1 \neq 1$. Let $h \in G_1$ be an element that does not commute with $g_1$. Conjugation by $(h,1,1,\ldots,1)$ shows that $(hg_1h^{-1},g_2,g_3,\ldots,g_k) \in K$, so  $(hg_1h^{-1}g_1^{-1},1,1,\ldots,1) \in K$. Hence if $K$ has a nontrivial projection to a factor then it intersects that factor in an infinite normal (hence nonabelian) subgroup. As $\pr(K)=l$, it intersects exactly $l$ factors.
\end{proof}

The following easy consequence of Lemma \ref{l:free_subgroup} is a variation on \cite[Lemma~6.3]{HW2}.

\begin{lemma}\label{l:product_rank_exact_sequences}
If $H$ is a finite-index subgroup of $G$ then $\pr(H)=\pr(G)$. If \[ 1 \to N \to G \to Q \to 1 \] is an exact sequence of groups then $\pr(G) \leq \pr(N) + \pr(Q)$.
\end{lemma}  

\medskip

\def\B{\mathcal{B}}

\noindent{\bf The product rank of $\Aut(F_N)$ and $\Out(F_N)$.}

Fixing a basis $\B = \{a_1,a_2, x_1,\ldots,x_{N-2}\}$ for $F_N$, one obtains a direct product of $2N-4$ free groups of rank $2$ in $\Aut(F_N)$ as follows.
For $i=1,\dots, {N-2}$ let $L_i$ be the subgroup consisting of automorphisms of the form
$[x_i\mapsto w x_i, \ x_j\mapsto x_j \, (j\neq i)]$, where $w$ is a word in the free group on $\{a_1, a_2\}$,
and let $R_i$ be the subgroup consisting of automorphisms of the form
$[x_i\mapsto x_i w, \ x_j\mapsto x_j \, (j\neq i)]$. Each $L_i$ and $R_i$ is a free group of rank $2$, and these subgroups
generate a direct product $D_\B= L_1\times R_1\times\dots\times L_{N-2}\times R_{N-2}< \Aut(F_N)$. As $D_\B$
contains no inner automorphisms, it injects into  $\Out(F_N)$. Theorem~6.1 of \cite{HW2} shows that $\Out(F_N)$ does not
contain a direct product of $2N-3$ nonabelian free groups if $N>2$, thus
 \[\pr(\Out(F_N))=2N-4\] 
 for $N\ge 3$. (Note that the virtual cohomological
dimension of $\Out(F_N)$, which gives an upper bound on product rank, is $2N-3$.) 

The conjugations of $F_N$ by $a_1$ and $a_2$ generate a further
 free subgroup $I(A)<\Aut(F_N)$ that commutes with $D_\B$. As $I(A)\cap D_\B$ is trivial,
 we get \[\pr (\aut(F_N)) \ge 2N-3,\]
and by Lemma \ref{l:product_rank_exact_sequences} we must have equality when $N\ge 3$.
Since $\aut(F_2)$ does not contain a direct product of two nonabelian free
groups \cite{gordon} (see also Corollary~\ref{c:only-tau} (2) below), we have equality in the case $N=2$ as well.

We summarize this discussion for later use:  

\begin{proposition}[\cite{HW2}, Theorem~6.1] \label{p:product_rank_aut_out}
For every $N \geq 2$ we have \[\pr (\aut(F_N)) =2N-3.\] For every $N \geq 3$ we have \[\pr(\Out(F_N))=2N-4.\]
\end{proposition}

Since $\Out(F_2)\cong {\rm{GL}}(2,\Z)$ is virtually free, $\pr(\Out(F_2))=1$.

\subsection{Splittings and their stabilizers}\label{s:splittings_background}

A {\em splitting} of a group $G$ is a minimal, simplicial left action on a tree. (The terminology comes from the fact that the
quotient graph of groups splits $G$ in terms of amalgamated free products and HNN extensions \cite{Serre}.) The splitting is said to be {\em free}
if all edge stablizers are trivial. Two splittings $T$ and $T'$ are deemed {\em equivalent} if there is a $G$-equivariant 
simplicial isomorphism from $T$ to $T'$. The trees that we consider 
are not allowed to have vertices of valence two. (The quotient graph of groups
may still have  vertices $v$ of valence two, in which case the vertex group $G_v$ will be nontrivial.) We say that $T'$ is a \emph{collapse} of $T$ if the action of $G$ on $T'$ is obtained by equivariantly collapsing a forest in $T$. Going in the opposite direction, we say that $T$ is a \emph{refinement} of $T'$ if $T'$ is a collapse of $T$.  Two splittings are said to be {\em compatible} if they have a common refinement.

We shall be concerned almost entirely with the case $G=F_N$.

Each vertex stabilizer of a free splitting of $F_N$ is a free factor. 
We work with the standard left action of $\Aut(F_N)$ on $F_N$.
There is then a right action  
of
$\Aut(F_N)$  on the set of all free splittings of $F_N$: the action of $\phi\in \Aut(F_N)$ sends $f: F_N\to {\rm{Isom}}(T)$ to 
$f\circ\phi$. This action respects equivalence classes of $F_N$-trees, and the inner automorphisms leave each equivalence class invariant. 
Thus there is an induced action of $\Out(F_N)$ on the set of equivalence classes of free splittings of $F_N$.
Stabilizers under this action have been studied extensively in the literature; the most general results (replacing $F_N$ with an arbitrary group and allowing more general splittings) appear in work of Bass--Jiang and Levitt \cite{BJ, Lev}. 

We write $\FS$ for the set of 
equivalence classes of free splittings of $F_N$. When there is no danger of ambiguity, we shall not distinguish between a free splitting $T$ and its equivalence class $[T]$. 

Unpacking the definitions,  we see that  $[T]\in \FS$ is fixed by an outer automorphism $\Phi  \in \out(F_N)$ if and only if for each representative $\phi \in \Phi$ there is a homeomorphism $f_{\phi}:T \to T$ such that 
\begin{equation}
f_\phi(gx)=\phi(g)f_\phi(x)
\end{equation} for all $x \in T$ and $g \in F_N$; in other words, $[T]$ is fixed by $\phi \in \Aut(F_N)$ if and only if there is an isomorphism from $T$ to itself that is  `$\phi$-twistedly equivariant'. The map $f_\phi$ is unique. (This is true, more generally, for stabilizers of minimal irreducible $G$-trees.)  If $\phi$ is conjugation by $g$, then $f_\phi(x) = gx$.  

We use $\Stab(T)$ to denote the stabilizer of $[T]$ in $\Out(F_N)$. There is a homomorphism \[  \Stab(T) \to \Aut(T/F_N) \] given by the left action of each outer automorphism on the $F_N$-orbits of edges and vertices in $T$. 
We call the kernel of this map $\Stab^0(T)$. (Here,
$T/F_N$ is the quotient graph, not the quotient graph of groups.)

We use $\ia$ to denote the kernel of the map \[\Out(F_N) \to \gl_N(\mathbb{Z}/3\mathbb{Z}) \] given by the action of $\Out(F_N)$ on $H_1(F_N,\mathbb{Z}/3\mathbb{Z})$. The analogous subgroup of the mapping class group consists of \emph{pure} mapping classes
(see Theorem 7.1.A of \cite{ivanov}) and behaves similarly; both groups are torsion-free and passing to them avoids a good deal of troublesome periodic behaviour. In this vein, we will require the following consequence of \cite[Theorem~3.1]{MR4089372}.

\begin{proposition}[\cite{HW2}, Lemma 2.6] \label{p:stab0}
Suppose that $G < \ia$. If the $G$-orbit of $T\in \FS$  is finite, then $G$ fixes $T$;
moreover $G < \Stab^0(T)$ and $G$ fixes every collapse of $T$.
\end{proposition}

If $T$ is a free splitting with one $F_N$-orbit of edges, we say that $T$ is a {\em one-edge splitting}. A one-edge splitting is \emph{nonseparating} if the quotient graph $T/F_N$ is a loop, and \emph{separating} otherwise.  The link
between free splittings and our study of direct products of free groups   
is the following extract from \cite[Theorem 6.1]{HW2}. The original statement of this theorem involves passing to a subgroup of finite index in $G$, but Proposition~\ref{p:stab0} shows that this is not necessary when one adds the assumption that $G$ is contained in $\ia$.

\begin{theorem}[\cite{HW2}, Theorem~6.1] \label{t:preserved_splitting}
If $G < \ia$ is a direct product of $2N-4$ nonabelian free groups, then $G$ fixes a one-edge nonseparating free splitting of $F_N$. 
\end{theorem}

\subsection{Automorphic lifts}\label{s:automorphic_lifts}

Subgroups that stabilize free splittings in $\Out(F_N)$ have the striking feature that they virtually lift to $\Aut(F_N)$ (see, for instance \cite{HM4}). To describe this lifting, we need some further notation. Given $G < \Out(F_N)$, we let $\tilde G$ denote the preimage of $G$ in $\Aut(F_N)$. We view $F_N$ as a subgroup of $\tilde G$ via the identification $g \mapsto \ad_g$ of $F_N$ with the inner automorphisms. The following lemma is well known (see, for example, Lemma~6.7 of \cite{BGH}).

\begin{lemma} \label{l:lifting_condition}
Let $T$ be a splitting of $F_N$ such that the action of $F_N$ on $T$ is irreducible. A subgroup $G < \Out(F_N)$ fixes $T$ if and only if the action of $F_N$ on $T$ extends to an action of $\tilde G$ on $T$.
\end{lemma}

\begin{proof}
As in Section~\ref{s:splittings_background}, $G$ fixes $T$ if and only if for each $\phi \in \tilde G$ there exists a unique 
isomorphism $f_\phi :T \to T$ such that $f_\phi(gx)=\phi(g)f_\phi(x)$. If the action $f:F_N \to \Isom(T)$ extends to an action $\tilde f: \tilde G \to \Isom(T)$ then $f_\phi:=\tilde f(\phi)$ suffices. Conversely, if $G$ fixes $T$ and $\phi, \psi \in \tilde G$ then the isomorphisms $f_\phi$ and $f_\psi$ satisfy  \begin{align*} f_\phi f_\psi(gx)&=f_\phi(\psi(g)f_\psi(x)) \\ &=\phi\psi(g)f_\phi f_\psi(x), \end{align*} so by uniqueness $f_\phi f_\psi = f_{\phi\psi}$, and $\phi \mapsto f_\phi$ gives the required action. 
\end{proof}

Understanding this extended action allows one to construct automorphic lifts.

\begin{proposition}[Existence of automorphic lifts]\label{p:automorphic-lift}
Let $G$ be a subgroup of $\ia$ that fixes a free splitting $T$ and let  $\tilde G<\aut(F_N)$ be its preimage, which acts on $T$.
For each edge $e$ in $T$, the stabilizer $\tilde G_e < \tilde G$ of $e$ is isomorphic to $G$. The isomorphism is given by the restriction 
to $\tilde G_e$ of the quotient map $\pi: \Aut(F_N) \to \Out(F_N)$.
\end{proposition}

\begin{proof}
As $T$ is a free splitting, no inner automorphism fixes $e$ and the map $\tilde G_e \to G$ is injective. 
To see that it is surjective, suppose $\Phi \in G$ and let $\phi \in \Phi$ be a representative of $\Phi$. 
Proposition~\ref{p:stab0} tells us that $G<\Stab^0(T)$ and hence  $\tilde G$ preserves the $F_N$-orbits of edges in $T$. 
Therefore $f_\phi (e)=ge$ for some $g \in F_N$ and \[ f_{\ad_g^{-1}\phi}(e)=f_{\ad_g^{-1}}f_\phi(e)=g^{-1}\cdot ge=e. \] It follows that $\ad_g^{-1}\phi$ is a representative of $\Phi$ in $\tilde G_{ e}$.
\end{proof}

\section{Direct products of free groups in $M_k(F_2)$ with maximal product rank} \label{s:slim_insert}

 We consider semidirect products of the form $M_k(\LL)=(\LL\times\dots\times \LL)\rtimes {\rm{Aut}}(\LL)$, where there are
 $k$ copies of the arbitrary group $\LL$ and the action  of ${\rm{Aut}}(\LL)$ is diagonal.  Writing elements of  $M_k(\LL)$
in the form $(g_1,\dots,g_k ; \phi)$, the group operation is 
$$(g_1,\dots,g_k ; \phi).(h_1,\dots,h_k ; \psi) = (g_1\phi(h_1),\dots,g_k\phi(h_k) ; \phi\circ\psi).$$
These groups arise naturally in many settings. For example, if $X=K(\LL,1)$ then the group of homotopy classes of 
homotopy equivalences $X\to X$ fixing $(k+1)$ marked points is isomorphic to $M_k(\LL)$. When
$\LL$ is free, these groups play an important role in the study of graph cohomology \cite{CHKV} and homology stability results for automorphism groups of free groups \cite{HV2}. We shall be concerned almost entirely with the case $\LL=F_2$.

Our interest in $M_k(F_2)$  
stems from the fact that if $T$ is the Bass--Serre tree of a collapsed rose with $N-2$ petals in the boundary of Outer space then, 
as we will see later, $\stab^0(T) \cong M_{2N-5}(F_2)$. Our purpose in this
section is to give a detailed description of the ways in which a direct product of $(k+1)$
nonabelian free groups can be embedded in $M_k(F_2)$. 

\subsection{Generalities} \label{s:generalities}

We distinguish between the $k$ visible copies of $\LL$ in $M_k(\LL)$ by writing
$M_k(\LL)=(\LL_1\times\dots\times \LL_k)\rtimes {\rm{Aut}}(\LL)$.
If $\LL$ has trivial centre, then $g\mapsto \ad_g$ defines an isomorphism from $\LL$ to the group of inner automorphisms
$\Inn(\LL)<{\rm{Aut}}(\LL)$. We claim that this gives rise to a natural embedding of the direct product   
$\LL^{k+1}\hookrightarrow M_k(\LL)$ with image
$$M^0_k(\LL) := (\LL_1\times\dots\times \LL_k)\rtimes \Inn(\LL),$$ where the last summand in $\LL^{k+1}$ maps
to  
$$J := \{(g^{-1},\dots,g^{-1}; \ad_g) \mid g\in \LL\} < M_k(\LL).$$
The existence of this embedding points to the fact $\LL_1\times\dots\times \LL_k<M_k(\LL)$  is not a characteristic subgroup, and the 
semidirect product decomposition defining $M_k(\LL)$ is less canonical than the short exact sequence
$$
1\to M^0_k(\LL) \to M_k(\LL) \to {\out}(\LL) \to 1.
$$

\begin{lemma}\label{l:extra-G} With the notation established above,
\begin{enumerate}
\item $(g_1,\dots,g_k, x) \mapsto (g_1x^{-1},\dots,g_kx^{-1}; \ad_x)$ defines an isomorphism $\LL^{k+1}\overset{\cong}\to M^0_k(\LL)$
with inverse $(g_1,\dots,g_k; \ad_x) \mapsto (g_1x,\dots,g_kx, x)$.
\item For $i=1,\dots,k$, there exists an involution $\alpha_i \in {\rm{Aut}}(M_k(\LL))$ that exchanges $\LL_i$
and $J$ while restricting to the identity on ${\rm{Aut}}(\LL)$ and $\LL_{\ell\neq i}$;
\item these $\alpha_i$ generate a copy of the symmetric group $\sym(k+1)\hookrightarrow {\rm{Aut}}(M_k(\LL))$, and
the embedding $\LL^{k+1}\to M^0_k(\LL)<M_k(\LL)$ is $\sym(k+1)$-equivariant.
\end{enumerate}
\end{lemma}

\begin{proof} For (1):
it is clear that these maps are mutually inverse and a straightforward calculation establishes that they are homomorphisms. 

To prove (2) and (3), one verifies that the formula
$$
\alpha_i : (g_1,\dots,g_k; \phi) \mapsto (g_1g_i^{-1},\dots,g_{i-1}g_i^{-1},\, g_i^{-1},\, g_{i+1}g_i^{-1},\dots,g_kg_i^{-1}; \ad_{g_i}\phi)
$$
defines an automorphism with the desired properties.
\end{proof}

\subsection{Fixed subgroups, centralisers, and Nielsen transformations}

We remind the reader that an automorphism $\tau \in \aut(F_2)$ is called a \emph{Nielsen transformation} if there is
a basis $\{x_1, x_2\}$ of $F_2$ such that \begin{equation}\label{e:Nielsen} \tau(x_1)=x_1x_2 \text{ and } \tau(x_2)=x_2. \end{equation}
Any two Nielsen transformations are conjugate in $\aut(F_2)$ and there are various ways to distinguish them from other automorphisms.  In $\out(F_2)\cong {\rm{GL}}(2,\Z)$, an automorphism represents a power of a Nielsen transformation if and only if its associated matrix has trace $2$ and determinant $1$. 

\begin{lemma}\label{l:allN}
If a subgroup $H<\aut(F_2)$ consists entirely of powers of Nielsen transformations, then it is cyclic.
\end{lemma}

\begin{proof}  
$H$ intersects  $\Inn(F_2)$ trivially and therefore injects into $\out(F_2)\cong {\rm{GL}}(2,\Z)$, where its image consists
entirely of elements that have trace $2$ and determinant $1$; let $M$ be such an element. We choose a basis so that 
the image of $H$ contains 
$
E_n= \begin{pmatrix}
1  & n \cr 0 & 1
\end{pmatrix}
$
for some non-zero integer $n$. If
$
M=
\begin{pmatrix}
a & b \cr c & d
\end{pmatrix},
$
 then $M E_n$  has trace $2$ only if $c=0$ and $a=d=1$. Hence the image of $H$ is contained in the cyclic subgroup of ${\rm{GL}}(2,\Z)$ generated by $E_1$.
\end{proof}

A characterisation of Nielsen
transformations that will be useful to us here concerns
the rank of fixed subgroups; this is due to Collins and Turner \cite{CT}.
Bestvina and Handel's solution to the Scott Conjecture \cite{BH} shows that the subgroup of   $\fix(\phi)<F_N$ 
fixed by an automorphism $\phi$
has rank at most $N$.  Theorem~A of \cite{CT} gives a complete description of the automorphisms with $\rk(\fix(\phi))=N$;
in the case $N=2$, these automorphisms are the powers of Nielsen transformations.

\begin{proposition}[\cite{CT}, Theorem~A] \label{p:CT} If $\phi\in \aut(F_2)$ fixes a non-cyclic subgroup of $F_2$, then $\phi$ is  a power of
a Nielsen transformation. If $\tau$ is the Nielsen automorphism given in Equation~\eqref{e:Nielsen} then the fixed subgroup of every nontrivial power of $\tau$ is $\langle x_1x_2x_1^{-1}, x_2 \rangle$.
\end{proposition}

\begin{corollary} \label{c:only-tau}
\begin{enumerate} 
\item If $\L< \Inn(F_2)$ is not cyclic, 
then its centraliser $C(\L)<\aut(F_2)$ is either trivial or else the cyclic subgroup generated by a Nielsen transformation.
\item $\rk_F(\aut(F_2))=1$.
\item Let $\phi,\psi\in\aut(F_2)$. If $\fix(\phi)\cap\fix(\psi)$ is not cyclic, then there is a unique Nielsen transformation $\tau^{\pm 1}$ such
that $\phi,\psi\in\<\tau\>$.
\item Let $\tau_0,\tau_1\in\aut(F_2)$ be Nielsen transformations. If $\<\tau_0\>\cap\<\tau_1\>\neq 1$, then $\tau_0=\tau_1^{\pm 1}$.
\end{enumerate}
\end{corollary}

\begin{proof} 
This proof is a variation on the argument used by Gordon \cite[Theorem~3.2]{gordon} to prove that $\rk_F(\aut(F_2))=1$.
If $\phi\in\aut(F_2)$ centralizes $\Lambda$, then $\phi(w)=w$ for all $\ad_w\in\L$, 
so  
$\phi$ is a power of a Nielsen transformation. 
Thus $C(\L)$ consists entirely of powers of Nielsen 
transformations, so it is cyclic, by Lemma \ref{l:allN}. If $\tau$
is a Nielsen transformation then $\fix(\tau) = \fix(\tau^p)$ for $p\neq 0$, so $\tau^p\in C(\L)$ implies $\tau\in C(\L)$. 
This proves (1). 

Since $\out(F_2)$ is virtually free, the centraliser of any nonabelian free subgroup is finite. So if there were a copy of $F_2\times F_2$ in
$\aut(F_2)$ then at least one of the factors, say $F_2\times 1$, would have to intersect $\Inn(F_2)$ non-trivially. The intersection would be normal in $F_2\times 1$, hence non-cyclic. On the other hand, the centraliser of this intersection would contain $1\times F_2$, contradicting (1). 
 
To prove (3), we apply (1) to  $  \{\ad_x \mid x\in\fix(\phi)\cap \fix(\psi)\}$.  And (3) implies (4).
\end{proof}

\begin{lemma} \label{c:commuting-Nielsen}
Let $\tau_0,  \tau_1 \in \Aut(F_2)$ be Nielsen transformations. If some non-zero powers of $\tau_0$ and $\tau_1$ commute,
then $\tau_0$ and $\tau_1$ commute, and $\tau_1= \ad_w\tau_0^{\pm 1}$ for some $w\in \fix(\tau_0)$.
\end{lemma}

\begin{proof}   By hypothesis, 
$[\tau_0^p,\tau_1^q]=1$ for some $p,q\neq 0$, that is  $(\tau_0^{-p} \tau_1\tau_0^{p})^q=\tau_1^q$. From Corollary \ref{c:only-tau}(4)
we deduce $\tau_0^{-p} \tau_1\tau_0^p=\tau_1$. Similarly, $(\tau_1^{-1} \tau_0\tau_1)^p=\tau_0^p$ implies
 $\tau_1^{-1}\tau_0\tau_1=\tau_0$. 
 
In the linear action of ${\rm{GL}}(2,\Z)\cong\out(F_2)$ on $\mathbb{R}^2$, the image of each Nielsen transformation $\tau$
has a unique eigenspace $V$ and $\tau$ generates  the pointwise stabiliser of $V$ in ${\rm{SL}}(2,\Z)$. 
As $\tau_0$ and  $\tau_1$ commute, their eigenspaces coincide.  
Thus, replacing $\tau_0$ by its inverse if necessary,  $\tau_1 = \ad_w \tau_0$.  
 Comparing $\tau_0\tau_1 =  \ad_{\tau_0(w)} \tau_0^2$ to $\tau_1\tau_0 =  \ad_{{w}} \tau_0^2$,
 we have $\tau_0(w)=w$.
\end{proof}

We remind the reader that two elements $\phi,\psi\in\aut(F_N)$ are defined to be {\em similar} if   there is  
an inner automorphism $\ad_y$ such that $\phi  = \ad_y\circ \psi\circ \ad_y^{-1} = \ad_{y\psi(y)^{-1}} \circ \psi$. 
As conjugacy preserves the set of (powers of) Nielsen transformations, so does similarity. Thus the following
lemma would be trivial if we were to replace ``only if" by ``if". 

\begin{lemma}\label{l:similar}
Let $\tau$ be  a Nielsen transformation and suppose $p\neq 0$. Then $\phi=\ad_w\tau^p$ is a 
power of a Nielsen transformation only if $w=y\tau^p(y)^{-1}$ for some $y\in F_2$ --
equivalently, $\phi = (\ad_y\circ \tau\circ \ad_y^{-1})^p$.
\end{lemma}

\begin{proof}  In the light of Proposition \ref{p:CT},
it suffices to argue that $\fix(\phi)$ is cyclic (possibly trivial) if $\phi$ is not similar to $\tau^p$, and this is a special case of the inequality 
 in \cite[Corollary 6.4]{BH}. 
\end{proof}

\subsection{A description of subgroups in $M_k(F_2)$ with maximal product rank}

\def\t{{\underline{\tau}}}
\def\F{F^k}

In the remainder of this section, we focus on the case $\LL=F_2$. To lighten the notation, we make the abbreviations $M_k=M_k(F_2)$
and 
$\F:=(\LL_1\times\dots\times \LL_k)$.  
We have a fixed identification of each $\LL_i$ with $F_2$,
with respect to which the action defining the semidirect product $M_k=\F\rtimes \aut(F_2)$ is diagonal. Note that,
for each index $i$, the subgroup $\LL_i^\tau < \LL_i$ fixed by $\tau\in\aut(\LL_i)=\aut(F_2)$ is the intersection of $\LL_i<\F$
with the centraliser in $M_k$ of $\t=(1,\dots,1;\tau)$. Likewise,  $J^\tau$ is the intersection of $J$ with the 
centraliser of $\t$, where $J$ is as defined in Section~\ref{s:generalities}.

\begin{remark} \label{r:p_rank_mk}
As  $\rk_F(\aut(F_2))=1$ and product rank is subadditive with respect to exact sequences (Lemma~\ref{l:product_rank_exact_sequences}), the exact sequence $1 \to F_2^k \to M_k(F_2) \to \aut(F_2) \to 1$ implies that $\rk_F(M_k(F_2)) \leq k+1$. The discussion in Section~\ref{s:generalities} gives an embedding of $F_2^{k+1}$ in $M_k(F_2)$, showing that the product rank is exactly $k+1$.
\end{remark}

We want to classify embeddings $D\hookrightarrow M_k$, where $D$ is a direct product
$$
D= D_1\times\dots\times D_{k+1}
$$
with each $D_i$  a nonabelian free group. (The $D_i$ need not be finitely generated.) We will be led to consider the images of the summands $D_i$ under the retraction $\pi: M_k\to   \aut(F_2)$
and the associated projection $\ov\pi: M_k\to   \out(F_2)$. Let $K_i:=\ker\pi|_{D_i}$.

\begin{theorem}\label{t:all-the-D} With the notation established above:
\begin{enumerate}
\item After permuting the indices, $[D_i,D_i]<\LL_i$ for $i=1,\dots,k$ and $[D_{k+1},D_{k+1}]<J$;
\item  if the centraliser $C_{M_k}(D) < M_k$ is non-trivial then there is a Nielsen transformation $\tau$ 
such that $C_{M_k}(D) =\<\t\>$, and after conjugating by an element of $F^k \times J$,
$$D< \LL_1^\tau\times\dots\times \LL_k^\tau\times J^\tau \times \<\t\>,$$
with $D_i< \LL_i\times \<\t\>$ for $i=1,\dots,k$ and $D_{k+1}< J^\tau\times\<\t\>$;
\item if $C_{M_k}(D)$ is  trivial, then $D$ satisfies one of the following conclusions, after conjugating  by an element of $F^k \times J$:
\begin{enumerate}
\item[{\bf i}.] $\ov\pi(D)=1$ and $D_i<\LL_i$ for $i=1,\dots,k$, while $D_{k+1}< J$ --  in particular,
$$D< \LL_1\times\dots\times \LL_k\times J;$$
\item[{\bf ii.}] $\ov\pi(D_i)=1$ for $i=1,\dots,k$ but $\ov\pi(D_{k+1})\neq 1$: in this case there is a Nielsen transformation $\tau$ such that 
 $D_i<\LL_i^\tau$ for $i=1,\dots,k$ and $D_{k+1}< \<J, \t\>$, so
$$D< \LL_1^\tau\times\dots\times \LL_k^\tau\times \< J,\t \>;$$
\item[\bf iii.] $\ov\pi(D_{k+1})= 1$ but there is a unique $j\le k$ such that $\ov\pi(D_j)\neq 1$:  in this case 
 $D_j<\<\LL_i, \t\>$ while $D_{k+1}< J^\tau$ and $D_i< \LL_i^\tau$ for $j\neq i\le k$, whence
$$D< \LL_1^\tau\times \dots\times \LL_{j-1}^\tau\times \<\LL_j,\tau\>\times \LL_{j+1}^\tau\dots\times \LL_k^\tau\times J^\tau.$$
\end{enumerate}
\end{enumerate}
\end{theorem}

The following implicit feature of the theorem warrants explicit mention.

\begin{addendum}
If $\ov{\pi}(D_i)$ is non-trivial for at least two of the summands $D_i<D$, then $C_{M_k}(D)=\<\tau\>$ for some
Nielsen transformation $\tau$.
\end{addendum}

\begin{remarks}\begin{enumerate}
\item The action of $\alpha_j\in \aut(M_k)$ interchanges cases 3({{\bf{ii}}}) and 3({{\bf{iii}}}). 
\item If $C_{M_k}(D)$ is non-trivial, $D$ might still conform to one of the descriptions in 3({{\bf{i}}}-{{\bf{iii}}}). 
\item In cases 3({\bf{{ii}}}) and 3({\bf{{iii}}}),  the subgroups $\<\LL_i,\t\>$ and $\<J,\t\>$
are isomorphic to $F_2\rtimes_\tau\Z$, a virtually special 3-manifold group in which free groups abound.
\item In case (2) there is visibly a unique maximal subgroup of the form $F^{k+1}\times\Z$ containing $D$ and in case 3({\bf{{i}}})
there is a  unique maximal subgroup of the form $F_2^{k+1}$.  We shall discuss the existence of maximal subgroups in the
other cases in Section \ref{s:maximal}
\end{enumerate} 
\end{remarks}

\subsection{The proof of Theorem~\ref{t:all-the-D}}

\begin{lemma}\label{l:one-only} Consider $D=D_1\times\dots\times D_{k+1}\hookrightarrow M_k = F^k\rtimes \aut(F_2)$.
\begin{enumerate}
\item There is a unique index $i_0$ such that $\pi|_{D_{i_0}}$ is injective.
\item There is a Nielsen transformation $\tau$ such that $\pi(D_j)<\<\tau\>$ for all $j\neq i_0$, and
\item $\ov{\pi}(D_{i_0})< \<[\tau]\>$. 
\item $\pi|_{D_{i_0}}^{-1}(\Inn(F_2)) < J$.
\end{enumerate}
\end{lemma}

\begin{proof} If $K_i=\ker\pi|_{D_i}$ is non-trivial, then it is a nonabelian free group. Since $K_1\times\dots\times K_{k+1}<\F$
and $\rk_F(\F)=k$, it follows that $\pi|_{D_{i}}$ is injective for at least one index $i_0$. 
A pair of commuting 
nonabelian free subgroups in any group generate their direct product, so $\rk_F(\aut(F_2))=1$ implies that $i_0$ is unique.
This proves (1). 

To lighten the notation, we assume $i_0=k+1$. Then $K_1\times\dots\times K_k<\F$ is a direct product of $k$ free nonabelian
free groups, and we can relabel the indices to assume $K_j<\LL_j$ for $j=1,\dots,k$. We fix a pair of non-commuting elements
$\und{u}_i,\und{v}_i\in K_i$ for $i=1,\dots,k$ and denote their non-trivial coordinates
 $u_i$ and $v_i$ respectively; for example, $\und{u}_1=(u_1,1\dots,1;1)$ and $\und{v}_1=(v_1,1\dots,1;1)$.

The remaining parts of the lemma require an analysis of the elements of $D_{k+1}$ -- let $g=(\omega_1,\dots,\omega_k; \phi)$ be
such. Since $D_{k+1}$ commutes with $K_1$, we have $g\und{u}_1 = \und{u}_1 g$,  hence
$$
(\omega_1 \phi(u_1), \omega_2,\dots, 1; \phi) = (u_1 \omega_1 , \omega_2,\dots, 1; \phi),
$$
whence $\omega_1 \phi(u_1) \omega_1^{-1} = u_1 $. Similarly, $\omega_1 \phi(v_1) \omega_1^{-1} =  v_1 $. Since $u_1$ and $v_1$ do not commute, $\fix (\ad_{\omega_1}\circ\phi)$ is not cyclic.   
Therefore, by Proposition \ref{p:CT}, either $\phi=\ad_{\omega_1}^{-1}$ or else  $\ad_{\omega_1}\circ\phi$ is a non-zero power of a Nielsen 
transformation. In the latter case, Lemma \ref{l:similar} tells us that $\phi$ itself must be a power of a Nielsen transformation, say $\phi=\tau_0^p$,
and $\omega_1=y_1\tau_0^p(y_1)^{-1}$ for some $y_1\in F_2$. 

Repeating this argument with $\{u_i,v_i\}$ in place of $\{u_1,v_1\}$, we see that either $\phi=\ad_{\omega_i}^{-1}$ for $i=1,\dots,k$
or else $\phi=\tau_0^p$ and $\omega_i$ has the form $y_i\tau_0^p(y_i)^{-1}$   for $i=1,\dots,k$. From the former case we deduce that
\begin{equation}\label{inJ}
\pi|_{D_{k+1}}^{-1}(\Inn(F_2)) < J,
\end{equation}
as claimed in (4).
From the latter case we deduce that $\ov{\pi}(D_{k+1})<\out(F_2)$ 
consists entirely of elements of trace 2 and determinant one. As in Lemma
\ref{l:allN}, this implies that $\ov\pi(D_{k+1})<\out(F_2)$ is cyclic, generated by $[\tau_0^r]$, say, where 
$\tau_0^r\in\pi(D_{k+1})$.
 
At this stage, we know that  $\pi(D_{k+1})<\aut(F_2)$ is a nonabelian free group and $\ov\pi(D_{k+1})<\out(F_2)$ is cyclic.
Thus $\Lambda:=\pi(D_{k+1})\cap\Inn(F_2)$ is not abelian, and Corollary \ref{c:only-tau} provides a Nielsen transformation
$\tau$ that generates the centralizer of $\Lambda$ in $\aut(F_2)$. As $\pi(D_j)$ commutes with $\Lambda$ when $j\le k$, part  (2) of the lemma  is proved. (If $\pi(D_j)=1$ for all $j\le k$, then we take $\tau=\tau_0$.)

If $\pi(D_j)=\<\tau^p\>$ for some $j\le k$ and $p\neq 0$, then $\tau^p\in \pi(D_j)$ commutes with $\tau_0^r\in\pi(D_{k+1})$.
Hence $\tau_0 = \ad_w\circ\tau$, by Lemma \ref{l:similar}, and (3) is proved.
\end{proof}

\noindent{\bf{Proof of Theorem \ref{t:all-the-D}} }With the lemma in hand, we may assume that for $i\le k$ we have $D_i=K_i<\LL_i$
or  $D_i = K_i\rtimes \<T_i\>$ with $K_i < \LL_i$ and $T_i=(t_{i1},\dots, t_{ik} ; \tau^{p_i})$, some $p_i\neq 0$. Also,
 $D_{k+1} <J$ or else $D_{k+1} = K_0\rtimes \< T_0\>$ with $K_0<J$ and
 $T_0=(t_{01},\dots, t_{0k} ; \tau_0^{p_0})$. Part (1) of the theorem follows.
\medskip

To prove (2) and the addendum, we assume that there are at least two indices  with $D_i = K_i\rtimes \<T_i\>$
and prove that this forces $D$ to be as described in part (2) of the theorem.
For clarity of exposition, we assume that this set of indices includes $\{1,2\}$. (The superficially-exceptional
case where one of the indices is $(k+1)$ is reduced to this case by applying one of the automorphisms $\alpha_i$.) 

If $j\neq 1$, then for every $\und{w}\in K_j$ we have $T_1\und{w}= \und{w}T_1$. The $j$-coordinate of $T_1 \und{w}$ is $t_{1j}\tau^{p_1}(w)$, whereas
the $j$-coordinate of $\und{w}T_1$ is $wt_{1j}$. Thus  $\tau^{p_1}(w) = t_{1j}^{-1}w t_{1j}$ for all $w\in K_j$; in other words
$f_{1j}:=\ad_{t_{1j}}\tau^{p_1}$ fixes $K_j$. 

Lemma \ref{l:similar} provides $y_{1j}\in F_2$ such that $$t_{1j} = y_{1j}\,\tau^{p_i}(y_{1j})^{-1},$$
so $f_{1j}=\ad_{y_{1j}}  \circ\tau^{p_1} \circ\ad_{y_{1j}}^{-1}.$
 Similarly, for $j\neq 2$ we obtain
 $f_{2j}=\ad_{y_{2j}}  \circ\tau^{p_2} \circ\ad_{y_{2j}}^{-1}$ fixing $K_j$.
 Corollary \ref{c:only-tau}(4) tells us that $f_{1j}$ and $f_{2j}$ are powers of a common Nielsen transformation, hence 
 $$\ad_{y_{1j}} \circ \tau  \circ\ad_{y_{1j}}^{-1} = \ad_{y_{2j}}  \circ\tau  \circ\ad_{y_{2j}}^{-1},$$
with $ y_{1j}y_{2j}^{-1} \in \fix(\tau)$. Taking powers and simplifying, this implies that
  for all $m\in\Z$, 
 \begin{equation}\label{e:equal}
 y_{1j}\tau^{m}(y_{1j})^{-1} = y_{2j}\tau^{m}(y_{2j})^{-1}.
 \end{equation} 
 We now conjugate $D$ by $\gamma=(y_{21}, y_{12}, y_{13},\dots, y_{1k}; 1)^{-1}$, noting that 
 $$
 (y_{21}, y_{12}, y_{13},\dots, y_{1k}; 1)^{-1} \ T_1 \ (y_{21}, y_{12}, y_{13},\dots, y_{1k}; 1) = (*, z_{12}, z_{13},\dots,z_{1k}; \tau^{p_1}),
 $$
 where for $j\ge 2$ we have
 $$
 z_{1j} = y_{1j}^{-1} t_{1j} \tau^{p_1}(y_{1j}) =  y_{1j}^{-1} y_{1j} \tau^{p_1}(y_{1j})^{-1} \tau^{p_1}(y_{1j}) =1.
 $$
 Likewise, 
  $$
 (y_{21}, y_{12}, y_{13},\dots, y_{1k}; 1)^{-1} \ T_2 \ (y_{21}, y_{12}, y_{13},\dots, y_{1k}; 1) = (z_{21}, *, z_{23},\dots,z_{2k}; \tau^{p_2}),
 $$
 where,
 using equation (\ref{e:equal}) to replace  $y_{2j} \tau^{p_2}(y_{2j})^{-1}$ by $y_{1j} \tau^{p_2}(y_{1j})^{-1}$, for $j\neq 2$ we have
  $$
 z_{2j} = y_{1j}^{-1} t_{2j} \tau^{p_2}(y_{1j}) =  y_{1j}^{-1} y_{2j} \tau^{p_2}(y_{2j})^{-1} \tau^{p_2}(y_{1j}) =
 y_{1j}^{-1} y_{1j} \tau^{p_2}(y_{1j})^{-1} \tau^{p_2}(y_{1j}) =1.
 $$
 An entirely similar argument applies to each index $i$ with $D_i = K_i\rtimes\<T_i\>$.
 
 Thus, after this conjugation (and abusing notation by identifying $D$ with $D^\gamma$), the $T_i$ have the form
 $$
 T_1 = (t_1,1,\dots,1;\tau^{p_1}),\ \ 
 T_2 = (1,t_2,1,\dots,1;\tau^{p_2}),\ \  
 T_i = (1,\dots,t_i,\dots,1; \tau^{p_i}).
 $$
 The commutation  $[T_1, K_j]=1$ now forces $K_j<\fix(\tau^{p_1})=\fix( \tau)$ for  $j\ge 1$ (including the case
 $K_j=D_j$). And  $[T_2, K_1]=1$ forces $K_1<\fix(\tau)$.
 Finally, the relations $[T_1,T_i]=1$ imply $t_j\in\fix(\tau)$ for $j=1,\dots,k$. In particular, $T_i \in \LL_i^\tau \times \<\t\>$, hence
 $D_i <eq \LL_i^\tau \times \<\t\>$ for $i\le k$. (The superficially-exceptional case $j=k+1$ can 
 again be handled by exchanging $D_{k+1}$ and some $D_i$ using $\alpha_i\in\aut(M_k)$.)
 This completes the proof of (2).
 
 \medskip
 It remains to consider what happens when $\ov{\pi}(D_i)\neq 1$ for at most one index $i$. Case 3({\bf{\rm{i}}}) is covered
 by Lemma \ref{l:one-only} and Case 3({\bf{\rm{ii}}}) can be reduced to 3({\bf{\rm{iii}}}) by applying the automorphism $\alpha_j$,
 so we address Case 3({\bf{\rm{iii}}}), taking $j=1$ for clarity. The proof in this case is a simplified version of the proof of (2):
 we have $D_j<\LL_j$ for $j=2,\dots,k$ and $D_{k+1}<J$, and after conjugating we may assume that 
 $D_1 = \LL_1\rtimes\<T_1\>$ where $T_1=(t_1,1,\dots,1;\tau^p)$ with $p\neq 0$. Again, the commutation  $[T_1, D_j]=1$ forces
 $D_j<\LL_j^\tau$ for $j\le k$ and $D_{k+1}<J^\tau$.
 \qed

\subsection{Extending by $\sym(k+1)$}

Theorem  \ref{t:all-the-D} describes embeddings into $M_k$ of direct products of $(k+1)$ free groups. The following proposition
shows that one gets no extra embeddings when the target is enlarged to $M_k\rtimes \sym(k+1)$, where the action in the
semidirect product is the same as Lemma \ref{l:extra-G}: the
transposition $(i\ k+1)\in\sym(k+1)$ acts as $\alpha_i$.

\begin{proposition}\label{p:augment} If $D$ is the direct product of $(k+1)$ nonabelian free groups, then 
the image of every embedding $D\hookrightarrow M_k\rtimes \sym(k+1)$ lies in $M_k\times 1$. Furthermore, if $\phi \in M_k\rtimes \sym(k+1)$ centralizes $D$, then $\phi$ lies in $M_k\times 1$.
\end{proposition}

\begin{proof} We identify $D$ with its image in $M_k\rtimes \sym(k+1)$.
There is no loss of generality in assuming that $D$ is finitely generated. Let $D^*<D$ be the subgroup obtained by replacing each
direct factor $D_i<D$ with the intersection of the kernels of all non-trivial homomorphisms from $D_i$ to (the abstract group)  $\sym(k+1)$. 
Then $D^* <  M_k \times 1$ and as in Theorem \ref{t:all-the-D}(1) we may assume that $[D^*_i, D^*_i]<\LL_i$ for $i\le k$
and $[D^*_{k+1}, D^*_{k+1}]<J$. The action of $D$ by conjugation on $D^*$ preserves each of the subgroups $[D^*_i, D^*_i]$.
In contrast, conjugation in $M_k\rtimes \sym(k+1)$ by any element of the form $(m; \sigma)$ will (if we write $J=\LL_{k+1}$) send $\LL_i$
to $\LL_{\sigma(i)}$; in particular it will not leave $[D^*_i, D^*_i]$ invariant if $\sigma(i)\neq i$. If $\phi$ centralizes $D$ then $\phi$ will also leave the groups $[D_i^*,D_i^*]$ invariant, so must also lie in $M_k \times 1$.
\end{proof}

\section{Stabilizers of collapsed roses and cages}\label{s:4}

 In this paper we restrict our attention to two simple examples of free splittings, which are \emph{collapsed roses} and \emph{cages}. We will make use of Bass--Serre theory, for which the standard references are Serre's book \cite{Serre} and the topological approach given by Scott and Wall \cite{SW}.

\subsection{The Bass--Serre tree of a collapsed rose}
 
A \emph{collapsed rose with $k$ petals} is a graph of groups decomposition of $F_N$ with a single vertex group and $k$ loops with trivial edge groups. The vertex group is necessarily isomorphic to a free factor $A\cong F_{N-k}$. We abuse notation slightly and refer to a free splitting $T$ of $F_N$ as a collapsed rose if the corresponding graph of groups is a collapsed rose. In the above notation, the vertex stabilizers of $T$ are the conjugates of the free factor $A$. A rose with one loop is simply a nonseparating free splitting. 

In $T$ we denote $v_A$ to be the vertex with stabilizer $A$ and pick representatives $e_1, \ldots, e_k$ of each orbit of edges that have initial vertex $v_A$. A \emph{stable letter} for $e_i$ is a choice of element $x_i$ such that $x_iv_A$ is the terminal vertex of $e_i$. Changing either the representative $e_i$ or the translating element gives possible stable letters of the form $ux_i^{\pm 1}v$, where $u,v \in A$. The free group $F_N$ is generated by $A$ and the stable letters $x_1, \ldots, x_k.$
 
\subsection{Stabilizers of collapsed roses in $\Out(F_N)$}\label{s:k-rose}

Let $T$ be a splitting of $F_N$ corresponding to a collapsed rose with $k$ petals. As in Section~\ref{s:splittings_background}, we let $\Stab(T)$ denote the stabilizer of $T$ in $\out(F_N)$. Letting $x_1, \ldots, x_k$ be a choice of stable letters for the rose, there is a finite subgroup $W_k$ generated by automorphisms $\sigma$ such that $\sigma$ is the identity when restricted to $A$ and for every stable letter $\sigma(x_i)=x_j^\epsilon$ for some $j$ and $\epsilon \in \{1,-1\}$. One can think of $W_k$ as the subgroup of $\out(F_N)$ given by permuting and inverting the petals of the rose, and $W_k$ is isomorphic to the semidirect product of $\mathbb{Z}/2\mathbb{Z}^k$ and the symmetric group $\sym(k)$. The homomorphism \[ F:\Stab(T) \to \aut(T/F_N)\] to the automorphism group of the rose is split surjective, and $W_k$ is a set of coset representatives of $\Stab^0(T)$ in $\Stab(T)$. 

\begin{proposition}\label{p:stab-description-rose} Let $T$ be a collapsed rose with $k$ petals. Let $\st$ and $\sto$ be the respective preimages of $\Stab(T)$ and $\Stab^0(T)$ in $\Aut(F_N)$. Let $b_A$ be a vertex of the tree with $F_N$-stabilizer $A$ and let $\st_A $ and $\sto_A$ be the respective subgroups of $\st$ and $\sto$ that fix $b_A$. An automorphism $\phi$ is an element of $\st_A$ if and only if there exists $w \in W_k$ such that:

\begin{enumerate}
 \item $\phi$ restricts to an automorphism of $A$ and,
 \item for each stable letter $x_i$ there exist $u_i, v_i \in A$ such that $\phi(x_i)=u_iw(x_i)v_i$.
\end{enumerate}
The group $\st_A$ is isomorphic to \[ M_{2k}(A) \rtimes W_k =(A^{2k}\rtimes \aut(A)) \rtimes W_k, \] via the isomorphism \[ \phi \mapsto (u_1^{-1}, \ldots, u_k^{-1}, v_1, \ldots, v_k; \phi|_A ; w). \] The group $W_k$ acts on $M_{2k}(A)$ as a subgroup of the group $\sym(2k+1)$ defined in Lemma \ref{l:extra-G}, and $\phi \in \sto_A$ if and only if $w=1$ in the above. 

If $e_j$ is the edge joining $b_A$ to $x_jb_A$ then the subgroup $\sto_{e_j} <  \sto_A$ fixing $e_j$ consists of automorphisms $\phi \in \sto_A$ such that $\phi(x_j)=x_jv_j$ (i.e. $u_j=1$); it is isomorphic to $M_{2k-1}(A)=A^{2k-1} \rtimes \aut(A)$.
\end{proposition}

\begin{proof}
 By the work on automorphic lifts in Section~\ref{s:automorphic_lifts}, $\phi \in \st$ if and only if there exists a $\phi$-twistedly equivariant map $f_\phi \co T \to T$ that preserves the $F_N$-orbits of edges and their orientations. Suppose that $\phi$ satisfies conditions $(1)$ and $(2)$ of the proposition. We define $f_\phi$ on the vertex set of $T$ by the map of cosets $gA \mapsto \phi(g)A$ (as the vertices of $T$ correspond to cosets of $A$ via Bass--Serre theory). Let $X=\{x_1,\ldots, x_k,x_1^{-1},\ldots, x_k^{-1}\}$. The cosets $gA$ adjacent to $1A$ are of the form $axA$, for $a \in A$ and $x \in X$. Therefore, under this correspondence between vertices and cosets of $A$, two vertices $gA$ and $hA$ span an edge in $T$ if and only if $g^{-1}h \in AXA$. As $\phi(AXA)=AXA$, the map $f_\phi$ preserves edges and so determines an isomorphism of $T$ which is clearly $\phi$-twistedly equivariant. Conversely, if $\phi \in \st_A$ then let $f_\phi: T \to T$ be a $\phi$-twistedly equivariant map fixing $b_A$. As $f_\phi(b_A)=b_A$, it follows that $\phi$ preserves $\Stab(b_A)=A$, so restricts to an automorphism of $A$. As $f_\phi$ preserves edges, using the same reasoning as above, we must have $\phi(AXA)=AXA$, from which it follows that $\phi$ must also satisfy (2).

The action of $\phi$ on the quotient graph $T/F_N$ is given by the last coordinate $w \in W_k$ in its decomposition, so $\phi \in \sto_A$ if and only if $w=1$. This gives an identification of $\sto_A$ with $M_{2k}(A)$ via the map \[ \phi \mapsto (u_1^{-1}, \ldots, u_k^{-1}, v_1, \ldots, v_k; \phi|_A). \] Conjugation by an element of $W_k$ acts on this group by a signed permutation of the $k$ pairs $\{u_i^{-1},v_i\}$, so $W_k$ is a subgroup of the group $\sym(2k+1)$ defined in Lemma~\ref{l:extra-G}.

An element $\phi \in \sto_A$ fixes the edge between $1A$ and $x_jA$ if and only if $\phi(x_j)A=x_jA$ (equivalently $f_\phi$ fixes this terminal vertex). As $\phi(x_j)A=u_jx_jA$, this happens if and only if $u_j=1$ and gives our description of $\sto_{e_j}$.
\end{proof}

Following Proposition~\ref{p:automorphic-lift}, each edge stabilizer $\sto_e$ is an automorphic lift of $\stab^0(T) < \out(F_N)$ to $\aut(F_N)$.

\subsection{Arc stabilizers in the Bass--Serre tree of a collapsed rose with $N-2$ petals}

In the proof of the Theorem~A, we will need to understand arc stabilizers for the action of $\st$ on the Bass--Serre tree of a collapsed rose with $N-2$ petals (with the notation of the previous section). In this case, each vertex stabilizer $A$ is isomorphic to $F_2$, so that $\sto_A \cong M_{2N-4}(F_2)$. Proposition~\ref{p:stab-description-rose} and Remark~\ref{r:p_rank_mk} together imply that the product rank of a vertex stabilizer of $T$ (in $\aut(F_N)$) is $2N-3$, and the product rank of an edge stabilizer is  $2N-4$. For later arguments, we need to show that the product rank drops further when one takes stabilizers of longer arcs in $T$.

\begin{lemma}\label{l:consecutive_edges}
Let $T$ be a collapsed rose with $N-2$ petals, and as above let $\st$ be the preimage of $\stab(T)$ in $\aut(F_N)$. If $\alpha$ is any edge path of length $\geq 2$ in $T$ then the pointwise stabilizer of $\alpha$ with respect to the $\st$--action on $T$ has product rank at most $2N-5$.
\end{lemma}

\begin{proof} It is enough to prove the result for a path $\alpha=\{e,e'\}$ of length two, and as product rank is preserved under passing to finite-index subgroups, we may pass to $G:=\sto$. We want to show that the intersection $G_e \cap G_{e'}$ of the edge stabilizers in $G$ has product rank at most $2N-5$. As above, suppose that $A \cong F_2$ is the subgroup of $F_N$ which stabilizes the vertex adjacent to $e$ and $e'$. Without loss of generality we may assume that $e$ is the edge between vertices with stabilizers $A$ and $x_1Ax_1^{-1}$ respectively. Then $G_e$ is the subgroup of $\Aut(F_N)$ that preserves $A$, sends $x_1$ to $x_1v_1$, and sends each $x_i$ to a word of the form $u_ix_iv_i$ for $2 \leq i \leq N-2$ (where as above each $u_i, v_i \in A$). Note that we can replace a stable letter $x_i$ with $ux_i^{\pm1}v$ if $2 \leq i \leq N-2$ and $u,v \in A$. We may also replace $x_1$ with $x_1v$ for some $v \in A$ without changing the description of $G_e$. Up to the above replacements, there are three possibilities for the stabilizer of the second vertex of $e'$: it is either of the form $x_jAx_j^{-1}$ for some $j$ satisfying $2 \leq j \leq N-2$, of the form $x_1^{-1}Ax_1$, or of the form $(wx_1)A(wx_1)^{-1}$ for some nontrivial $w \in A$. In the first case, we see that $G_e \cap G_{e'}$ consists of automorphisms $\phi$ preserving $A$ with the added restriction that both $u_j=1$ and $u_1=1$ in the above notation. Hence the intersection decomposes as an exact sequence \[ 1 \to A^{2k-2} \to G_e \cap G_{e'} \to \Aut(A) \to 1, \] where $k=N-2$. The product rank of the kernel is then $2N-6$ and the product rank of $\aut(A)=\aut(F_2)$ is one, so the product rank of the intersection is at most $2N-5$. 

In the second case where the terminal vertex of $e'$ has stabilizer $x_1^{-1}Ax_1$ a similar argument applies. One sees that $G_e \cap G_{e'}$ is the subgroup of the stabilizer of $v_A$ which fixes $x_1$ (hence $u_1=v_1=1$), and we have the same exact sequence as above.

In the final case, the intersection $G_e \cap G_{e'}$ is given by the automorphisms in $G_e$ which also fix the subgroup $(wx_1)A(wx_1)^{-1}$. Note that \[\phi((wx_1)A(wx_1)^{-1})=\phi(w)\phi(x_1Ax_1^{-1})\phi(w)^{-1}=\phi(w)x_1Ax_1^{-1}\phi(w)^{-1} \] as elements of $G_e$ preserve $x_1Ax_1^{-1}$. If $\phi$ is also an element of $G_{e'}$ then \[ (wx_1)A(wx_1)^{-1}= \phi(w)x_1Ax_1^{-1}\phi(w)^{-1}. \] This implies that $\phi(w)=w$ as $w \in A$. Therefore in the exact sequence \[ 1 \to A^{2N-5} \to G_e \to \Aut(A) \to 1, \] the intersection $G_e \cap G_{e'}$ projects to a subgroup of $\Aut(A)$ fixing the element $w$. However, parts 1 and 3 of Lemma~\ref{l:one-only} imply that a subgroup of $G_e \cong A^{2N-5} \rtimes \aut(A)$ of maximal product rank projects to a subgroup of $\aut(A)$ containing a nonabelian subgroup of inner automorphisms, and therefore cannot fix any $w \in A$. It follows that $G_e \cap  G_{e'}$ has product rank at most $2N-5$.
\end{proof}

\begin{proposition} \label{p:fixed_point}
Let $T$ be a collapsed rose with $N-2$ petals, and let $D$ be a direct product of $2N-3$ nonabelian free groups in $\st < \aut(F_N)$. Then $D$ has a unique global fixed point $v$ in $T$. The normalizer of $D$ in $\st$ also fixes $v$.
\end{proposition}

\begin{proof}
If a fixed vertex of $D$ exists then it is unique as stabilizers of edges have product rank $2N-4$ in $T$. By uniqueness, such a fixed point will also be invariant under the normalizer of $D$ in $\st$. We are left with the matter of proving such a fixed point exists.

Let $D=D_1 \times D_2 \times \cdots \times D_{2N-3}$. Firstly, suppose that some factor, for instance $D_{2N-3}$, contains a hyperbolic element $g$ with respect to the action on $T$. As each factor $D_i$ for $i < 2N-3$ commutes with $g$, the direct product $D_1 \times \cdots \times D_{2N-4}$ acts on the axis $A_g$ of $g$ preserving the orientation. The commutator subgroup $D_i'$ of each $D_i$ acts trivially on this axis, and we attain a direct product of $2N-4$ nonabelian free groups fixing a line in $T$. This contradicts Lemma~\ref{l:consecutive_edges}. We may therefore assume that each factor $D_i$ consists of elliptic elements. Any product of commuting elliptic elements is also elliptic, so every element of $D$ is elliptic. If $D$ does not have a global fixed point, then as in \cite[I.6.5, Exercise~2]{Serre} the group $D$ is not finitely generated and fixes an end of $T$. Every element of $D$ fixes a half-line towards this end, so that we can find a finitely generated subgroup $D_1' \times D_2' \times \cdots \times D_{2N-3}'$ with each $D_i'$ nonabelian fixing a half-line in $T$, which again contradicts Lemma~\ref{l:consecutive_edges}. Hence $D$ has a global fixed point in $T$.
\end{proof}

\subsection{Stabilizers of cages}\label{s:k-cage}

A splitting $T$ of $F_N$ is a \emph{cage} if $T$ is a free splitting and the quotient graph $T/F_N$ is isomorphic to a cage (a graph with two vertices and no loop edges). Suppose that $T/F_N$ is a cage with $k$ edges. Pick adjacent vertices $v$ and $w$ in $T$ with stabilizers $A$ and $B$. Let $e$ be the edge from $v$ to $w$ and let $e_1, \ldots, e_{k-1}$ be edges based at $v$ representing the other $k-1$ edge orbits in $T$. If $x_i$ is an element that takes the terminal vertex of $e_i$ to $w$, then $F_N$ is generated by $A$, $B$, and $x_1, \ldots, x_{k-1}$.

As above, take $G=\stab^0(T)$ and $\tilde G$ its preimage in $\aut(F_N)$. The stabilizer $\tilde G_e$ of $e$ gives an automorphic lift of $G$. As in the proof of Proposition~\ref{p:stab-description-rose}, every element $\Phi \in \Stab^0(T)$ has a unique representative $\phi \in \tilde G_e$ such that \begin{itemize}
\item $\phi$ restricts to an automorphism of $A$ and $B$.
\item For each $x_i$ we have $\phi(x_i)=a_ix_ib_i$ for some $a_i \in A$ and $b_i \in B$.
\end{itemize}

It follows that $\Stab^0(T)$ fits in the exact sequence

\[ 1\to A^{k-1} \times B^{k-1} \to \Stab^0(T) \to \aut(A) \times \aut(B) \to 1.\]

Furthermore, this sequence splits so that $\Stab^0(T)$ is a semidirect product of $\aut(A) \times \aut(B)$ with $A^{k-1} \times B^{k-1}$, where $\aut(A)$ acts diagonally on $A^{k-1}$ and trivially on $B^{k-1}$, and the action of $\aut(B)$ on $A^{k-1} \times B^{k-1}$ is given in the same fashion. In the proof of Proposition~\ref{pr:compatible}, we will use this decomposition to calculate $\rk_F(\Stab^0(T))$.

\section{Fixed splittings of subgroups with maximal product rank} \label{s:fix}

In this section, we prove the following proposition:

\begin{proposition} \label{pr:compatible}
Suppose that $G < \ia$ and $\pr(G)=2N-4$. Any splitting fixed by $G$ has one $F_N$-orbit of vertices (i.e. is a collapsed rose). Any two free splittings fixed by $G$ are compatible.
\end{proposition}

Proposition~\ref{pr:compatible} follows from two lemmas that will be proved below. Lemma~\ref{l:boundary_sphere_lemma} asserts that if a subgroup $G < \ia$ fixes two incompatible free splittings, then $G$ fixes a free splitting with at least two $F_N$-orbits of vertices. Lemma~\ref{l:counting} then shows that such a splitting cannot be fixed by a group with maximal product rank (in other words, subgroups with maximal product rank can only fix collapsed roses). We delay the full proof for the time being to first state two consequences of this proposition.

We say that a free splitting $T$ is $G$-unrefinable if it admits no $G$-invariant refinement that is a free splitting. We have the following useful corollary:

\begin{corollary} \label{c:max_unrefinable}
Suppose $G < \ia$ and $\pr(G)=2N-4$.
\begin{itemize}
 \item The group $G$ fixes a unique $G$-unrefinable free splitting. If this splitting is nontrivial then it is a collapsed rose.
 \item If $T$ is the unique $G$-unrefinable collapsed rose fixed by $G$, then $T$ is also fixed by the normalizer $N_{\Out(F_N)}(G)$ of $G$ in $\Out(F_N)$.
\end{itemize}
\end{corollary}

\begin{proof}
Take $X$ to be the set of all one-edge free splittings preserved by $G$. These are all compatible by Proposition~\ref{pr:compatible}, so $X$ is finite and there is a common refinement $T$ collapsing to every element of $X$ (\cite[Theorem~5.16]{MR2032389} or \cite[Proposition~A.17]{GL-jsj}). Also by Proposition~\ref{pr:compatible}, the tree $T$ is a collapsed rose. Any other $G$-invariant free splitting is a common refinement of a subset of $X$ (this follows from the fact that $G$ is contained in $\ia$ and Proposition~\ref{p:stab0}: if $G$ fixes a free splitting $T'$ then $G$ also fixes all of the one-edge splittings to which $T'$ collapses). It follows that $T$ is the unique $G$-unrefinable, $G$-invariant free splitting. By uniqueness, $T$ is invariant under the normalizer of $G$ in $\out(F_N)$.
\end{proof}

\begin{corollary} \label{c:max_rose}
Let $G < \ia$ be such that $\rk_F(G)=2N-4$. If $G$ fixes the Bass--Serre tree $T$ of a collapsed rose with $N-2$ petals, then this rose is unique and the normalizer $N(G)$ of $G$ in $\Out(F_N)$ also fixes $T$.
\end{corollary}

\begin{proof}
A $G$-invariant refinement of such a rose would have to have $N-1$ or $N$ petals. However it follows from Proposition~\ref{p:stab-description-rose} that the stabilizer of a collapsed rose with $N-1$ petals is virtually abelian and the stabilizer of a rose with $N$ petals lies in the interior of Outer space and is finite. Hence if $G$ fixes a collapsed rose with $N-2$ petals then this rose is $G$-unrefinable and invariant under the normalizer of $G$ in $\out(F_N)$.
\end{proof}

We move on to the required lemmas.

\begin{lemma}\label{l:boundary_sphere_lemma}
If $G < \ia$ fixes incompatible free splittings $T$ and $T'$ then  $G$ fixes a free splitting which contains at least two orbits of vertices.
\end{lemma}

\begin{proof}
As refinements of incompatible trees are also incompatible, we may assume that the two splittings $T$ and $T'$ are incompatible and $G$-unrefinable. Let $\tilde G$ be the preimage of $G$ in $\aut(F_N)$. By Proposition~6.2 of \cite{GHmeasure}, any two $G$-unrefinable free splittings belong to the same deformation space when viewed as $\tilde G$-trees (equivalently, $T$ and $T'$ have the same set of elliptic subgroups with respect to the $\tilde G$-action). However, if there are incompatible trees in a deformation space then there must be trees with at least two orbits of vertices. We give a brief explanation via folding: As $T$ and $T'$ belong to the same deformation space, there exists a $\tilde G$ equivariant map $f:T \to T'$ taking edges in $T$ to (possibly trivial) edge paths in $T'$ (see, e.g., \cite{MR2319455}).  The map $f$ decomposes as a collapse $f_0:T \to T_0$ followed by a morphism $f_1: T_0 \to T_1$ that does not collapse edges. The map $f_1$ is nontrivial as $T$ and $T'$ are incompatible, which implies there exist edges $e \neq e'$ based at the same vertex in $T_0$ such that the edge paths $f(e)$ and $f(e')$ have a common initial egde (if $f_1$ were locally injective, then $f_1$ would be an isomorphism). Partially folding $e$ and $e'$ along a small initial segment gives a new $\tilde G$ tree $T_1$ with an extra $\tilde G$-orbit of vertices at the fold. There is an induced morphism $f_2 : T_1 \to T'$, which implies that any edge stabilizer of $T_1$ fixes a nondegenerate arc in $T'$. Hence $F_N < \tilde G$ has no nontrivial edge stabilizer in $T_1$, and $T_1$ is a $G$-invariant free splitting with at least two orbits of edges.
\end{proof}

We now show that the stabilizer of any splitting with more than one orbit of vertices does not have maximal product rank. Proposition~\ref{pr:compatible} follows immediately.

\begin{lemma}\label{l:counting}
Let $T$ be a splitting of $F_N$ which contains at least two $F_N$-orbits of vertices. Then $\pr(\stab_{\Out(F_N)}(T)) \leq 2N-5$.
\end{lemma}

\begin{proof}
Suppose for a contradiction that $\rk_F(\stab(T))=2N-4$. Then the intersection of $\stab(T)$ with the finite index subgroup $\ia < \out(F_N)$ also contains a direct product $D$ of $2N-4$ nonabelian free groups, and $D$ preserves every collapse of $T$ (Proposition~\ref{p:stab0}). We may therefore replace $T$ with a collapse $T'$ that has exactly two orbits of vertices $v$ and $w$.  Collapsing all loop edges in the quotient graph then gives us a cage $S$ with $k \geq 1$ edges (allowing the degenerate case of a one-edge separating splitting). As $D < \ia$, we have $D < \Stab^0(S)$. If $A$ and $B$ are free factor representatives of the vertex stabilizers in $S$ then, as in Section~\ref{s:k-cage}, the group $\Stab^0(S)$ decomposes as a short exact sequence \[ 1\to A^{k-1} \times B^{k-1} \to \Stab^0(S) \to \aut(A) \times \aut(B) \to 1.\] Suppose $A$ and $B$ are of rank $n$ and $m$ respectively, so that $N = n+m+(k-1)$. If $A$ and $B$ are both noncyclic, then \[ \pr(\aut(A))+\pr(\aut(B))= (2n-3) + (2m-3) =2N-2k-4 \] and \[ \pr(A^{k-1} \times B^{k-1}) = 2k-2. \] Hence $\pr(\stab^0(S)) \leq (2N-2k-4 ) +  2k-2 =2N-6 $. If both $A$ and $B$ are cyclic (possibly trivial), then $\stab^0(S)$ is virtually abelian, so $\pr(\stab^0(S))=0$. Suppose $A$ is nonabelian and $B$ is cyclic. Then $\pr(A^{k-1} \times B^{k-1})=k-1$. If $B$ is trivial then $A$ is rank $N-k+1$ and the vertex $w$ has valence at least 3, so that $k \geq 3$. Then $\pr(\aut(A) \times \aut(B))=\pr(\aut(A)) \leq 2(N-k+1)-3$. Hence \[ \pr(\Stab^0(S)) \leq 2(N-k+1)-3 + k-1 = 2N -2 -k \leq 2N-5. \] Similarly, if $B$ is infinite cyclic then $A$ is rank $N-k$ and the same computation gives \[ \pr(\Stab^0(S)) \leq 2N-4-k \leq 2N-5. \]
\end{proof}

\section{Completing the proofs of theorems A and B} \label{s:main_thm}

We are now armed with enough knowledge to complete the proof of Theorem~A.

\begin{theorem} [Theorem~A]
Let $N \geq 3$ and suppose \[D=D_1 \times D_2 \times \cdots \times D_{2N-4} < \out(F_N) \] is a direct product of $2N-4$ nonabelian free groups. Then, in the boundary of Outer space, there is a unique collapse rose with $N-2$ petals that is fixed by $D$.  \end{theorem}

\begin{remark}\label{r:normalizer}
If such a collapsed rose exists then it is unique by Corollary~\ref{c:max_rose}, and therefore is also invariant under the normalizer $N_{\Out(F_N)}(D)$ of $D$. This will be very useful in the induction step of the proof below.
\end{remark}

\begin{proof}
We first reduce to the case where $D < \ia$. Let $D_i'=D_i \cap \ia$, and let \[ D' = D_1' \times D_2' \times \cdots D_{2N-4}'. \] Then $D'$ is a finite-index, normal subgroup of $D$, so if $D'$ fixes a collapsed rose with $N-2$ petals, so does $D$ (as $D < N(D')$, this follows from Remark~\ref{r:normalizer}). We can therefore assume that $D < \ia$.

When $N=3$, a rose with $N-2$ petals is simply a nonseparating free splitting.  By Theorem~\ref{t:preserved_splitting}, a direct product of two free groups $D=D_1 \times D_2$ in $\iat$ fixes such a splitting. 

For the inductive step, we will take the preimage $\tilde D$ of $D$ in $\aut(F_N)$ and build an action of $\tilde D$ on a tree $T$ such that the restriction to the inner automorphisms $F_N < \aut(F_N)$ is a free splitting with $N-2$ $F_N$-orbits of edges. This splitting is then $D$-invariant (Lemma~\ref{l:lifting_condition}) and is necessarily an $N-2$ rose by Corollary~\ref{c:max_unrefinable}. We build the tree as a \emph{graph of actions} (see \cite{Lev2}). A graph of actions for a group $H$ consists of the following data:

\begin{itemize}
 \item A graph of groups decomposition of $H$.
 \item For every vertex $u$ in the graph of groups, a tree $T_u$ equipped with an action of the vertex group $H_u$.
 \item For every oriented edge $e$ with terminus $u$, a fixed point $x_e$ of $H_e < H_u$ in $T_u$.
\end{itemize}

Given a graph of actions where the vertex trees are all simplicial, if $S$ is the Bass--Serre tree of the underlying graph of groups then one obtains a refinement $T$ of $S$ from the graph of actions. There is a natural collapse map $T \to S$, where the preimage of a vertex $\tilde u \in S$ is a copy of $T_u$. The points $x_e$ are used as gluing instructions for the endpoints of the edges of $S$ into the trees $T_u$. We will take $S$ to be any one-edge nonseparating free splitting invariant under $D$, which exists by Theorem~\ref{t:preserved_splitting}. Let $A \cong F_{N-1}$ be a free factor whose conjugates form the vertex stabilizers of $S$, and let $v$ be the vertex fixed by $A$. Let $D_A$ be the image of $D$ in $\out(A)$. By considering product rank in the short exact sequence \[ 1 \to A \times A \to \stab^0(S) \to \Out(A) \to 1, \] after reordering the factors we must have: \[D_A \cong (D_1 \times D_2)/N \times D_3 \times \cdots \times D_{2N-4},\] where $N$ is a normal subgroup of $D_1 \times D_2$. In particular, $D_A$ is in the normalizer of a direct product of $2N-6$ free groups in $\Out(A)$. By induction and Remark~\ref{r:normalizer}, we know that $D_A$ fixes a collapsed rose $T_A$ with $N-3$ petals. Let $\tilde D_A$ be the preimage of $D_A$ in $\aut(A)$. As $\tilde D_v$ preserves $A$, the vertex group $\tilde D_v$ acts on $T_A$ via the projection $\tilde D_v \to \aut(A)$. In order to show that we can define a graph of actions, we need to take an edge $e$ adjacent to $v$ and check that $\tilde D_e$ has a fixed point with respect to its action on $T_A$. We let $f: \tilde D_e \to \tilde D_A$ be the map factoring through $\tilde D_v$ that determines this action. We have the following commutative diagram.

\[
\xymatrix{
\tilde D_e \ar@/_/[ddr]_\cong \ar@/^1pc/[drr]^f \ar@{^{(}->}[dr]\\
&\tilde D_v \ar[d] \ar[r] & \tilde D_A\ar[d]  &{}\save[]-<0.1cm,0cm>* {< \aut(A)}
 \restore \\
&D \ar[r] & D_A &  {}\save[]-<0.1cm,0cm>* {< \out(A)}
 \restore \\ }
\]

The map from $\tilde D_e$ to $D$ is an isomorphism as $S$ is a free splitting, so that $\tilde D_e$ is an automorphic lift as described in Section~\ref{s:automorphic_lifts}. Without loss of generality, we can take $e$ corresponding to the stable letter $x_1$, so that the stabilizers of its endpoints are $A$ and $x_1A x_1^{-1}$. As in Section~\ref{s:k-rose}, every automorphism $\phi \in \tilde D_e$ restricts to an automorphism $\phi_A$ of $A$ and satisfies $\phi(x_1)=x_1a$ for some $a \in A$. It follows that the kernel of $f$ is a free group generated by these right transvections, so that after rearranging the factors of $\tilde D_e \cong D$, we have \[ \image(f) \cong D_1/ N_1 \times D_2 \times \cdots D_{2N-4},\] for some normal subgroup $N_1$ of $D_1$. Therefore $\image(f)$ is contained in the normalizer of a direct product of $2N-5=2(N-1)-3$ free groups in $\aut(A)$, so by  Proposition~\ref{p:fixed_point}, the group $\image(f)$ has a fixed point $x_e$ in  $T_A$.

It follows that $\tilde D$ admits a graph of actions with defining tree Bass--Serre tree $S$ and vertex tree $T_A$, so that the refinement of $S$ determined by this graph of actions is a free splitting of $F_N$ with $N-2$ orbits of edges.
\end{proof}

\begin{theorem} [Theorem~B]
Let $N \geq 3$ and suppose $D<\aut(F_N)$ is a direct product of $2N-3$ nonabelian free groups. Then,
the image of $D$ in $\out(F_N)$ fixes a unique collapsed rose with $N-2$ petals, and $D$ acts on the Bass--Serre tree of this collapsed rose with a unique global fixed point. 
\end{theorem}

\begin{proof}
If $D$ is a direct product of $2N-3$ nonabelian free groups in $\aut(F_N)$, then the image $\bar{D}$ of $D$ in $\out(F_N)$ is contained in the normalizer of a direct product of $2N-4$ nonabelian free groups (as the kernel of the map $D \to \bar{D}$ is free and normal, so contained in exactly one factor by Lemma~\ref{l:free_subgroup}). Hence $\bar{D}$ fixes a collapsed rose $T$ with $N-2$ petals in the boundary of Outer space, and that $D$ acts on this tree. Proposition~\ref{p:fixed_point} then states that $D$ has a unique global fixed point with respect to the action on $T$.
\end{proof}

\section{Algebraic descriptions of the direct products and their centralizers}\label{s:applications} 

In this section we will prove Theorems~C and D from the introduction, where the relevant notation was also established. We also prove an $\out(F_N)$-version of Theorem~C in Theorem~\ref{t:direct-products-out}.

\begin{theorem}[Theorem~C]
Let $N \geq 3$ and let $D < \aut(F_N)$ be a direct product of $2N-3$ nonabelian free groups. Then a conjugate of $D$ is contained in one of the following subgroups.
\begin{itemize}
\item $L_1 \times \cdots \times L_{N-2} \times R_1 \times \cdots \times R_{N-2} \times I(A)$
\item $L_1^\tau \times \cdots \times L_{N-2}^\tau \times R_1^\tau \times \cdots \times R_{N-2}^\tau \times I(A)^\tau \times \langle \tau \rangle $
\item $\langle \tau, L_1 \rangle \times L_2^\tau \times \cdots L_{N-2}^\tau \times R_1^\tau \times \cdots R_{N-2}^\tau \times I(A)^\tau$
\item $L_1^\tau \times \cdots \times L_{N-2}^\tau \times R_1^\tau \times \cdots \times R_{N-2}^\tau \times \langle I(A), \tau \rangle $
\end{itemize}
\end{theorem}

\begin{proof}
By Theorem~\ref{t:B}, there exists an action of $D$ on the Bass--Serre tree $T$ of a collapsed rose with $N-2$ petals such that $D$ has a (unique) global fixed point. Up to conjugation, we may assume that the collapsed rose has vertex group $A=\langle a_1, a_2 \rangle$ and stable letters $x_1, \ldots, x_{N-2}$, and $D < \st_A$ (i.e. the vertex fixed by $D$ in $T$ has $F_N$--stabilizer equal to $A$). By Proposition~\ref{p:stab-description-rose}, we have 
\[ \st_A=M_{2N-4}(A) \rtimes W_{N-2}, \]
where $W_{N-2}$ acts as a subgroup of the group $\sym{(2N-3)}$ of automorphisms of $M_{2N-4}(A)$ defined in Lemma~\ref{l:extra-G}. By Proposition~\ref{p:augment}, the projection of $D$ to $W_{N-2}$ is trivial, so that $D$ is contained in $\sto_A=M_{2N-4}(A)$. Under the isomorphism between $\sto_A$ and $M_{2N-4}(A)$ given in Proposition~\ref{p:stab-description-rose}, the groups $L_i$ and $R_i$ are taken to factors in the $A^{2N-4}$ subgroup of $M_{2N-4}(A)$, the inner automorphisms by elements of $A$ are taken to $J$, and $\tau$ is taken to $\underline{\tau}$ (using the notation of Section~\ref{s:slim_insert}). The possibilities for $D$ are then given by Theorem~\ref{t:all-the-D}.
\end{proof}

In the following, we blur the distinction between $L_i$ and $R_i$  and their (isomorphic) images in $\out(F_N)$. 

\begin{theorem}\label{t:direct-products-out}
Let $N \geq 3$ and let $D < \out(F_N)$ be a direct product of $2N-4$ nonabelian free groups. Then a conjugate of $D$ is contained in one of the following subgroups.
\begin{itemize}
\item $L_1 \times \cdots \times L_{N-2} \times R_1 \times \cdots \times R_{N-2}$
\item $L_1^\tau \times \cdots \times L_{N-2}^\tau \times R_1^\tau \times \cdots \times R_{N-2}^\tau \times \langle [\tau] \rangle $
\item $\langle [\tau], L_1 \rangle \times L_2^\tau \times \cdots L_{N-2}^\tau \times R_1^\tau \times \cdots R_{N-2}^\tau$
\end{itemize}
\end{theorem}

\begin{proof}
We apply Theorem~A to find a $D$-invariant collapsed rose with $N-2$ petals in the boundary of Outer space. We then conjugate $D$ so that this rose $T$ is the one given by $A$ and $x_1, \ldots, x_{N-2}$. In this case it is more natural to see $\stab^0(T)\cong M_{2N-5}(A)$ decomposing as the exact sequence:
\[ 1 \to L_1 \times \cdots \times L_{N-2} \times R_1 \times \cdots \times R_{N-2} \to \Stab^0(T) \to \Out(A) \to 1, \]
so that $\stab(T) =\stab^0(T) \rtimes W_{N-2}$ and $W_{N-2}$ acts by signed permutations of the $(R_i,L_i)$ pairs of factors in the kernel of this exact sequence. Automorphic lifting (see the proof of Proposition~\ref{p:automorphic-lift}) tells us that $\stab^0(T)\cong \sto_e$ for any edge $e$ in the tree. If $e$ is the edge between the vertices corresponding to the cosets $1A$ and $x_1A$ in the Bass--Serre tree, then in Proposition~\ref{p:stab-description-rose} we have seen that each $\phi \in \sto_e$ preserves $A$, maps $x_1 \mapsto x_1v_1$ and for $2 \leq i \leq N-2$ maps $x_ i \mapsto u_ix_iv_i$ with the $u_i,v_i \in A$. This gives an isomorphism between $\stab^0(T)= \sto_e$ and $M_{2N-5}(A)$ via \[ \Theta(\phi) = (u_2^{-1},\ldots, u_{N-2}^{-1},v_1, \ldots v_{N-2} ; \phi|_A ) \]
Although the images of $L_2, \ldots, L_{N-2}$ and $R_1, \ldots, R_{N-2}$ are easy to see under $\Theta$, the `natural' representatives of the elements of $L_1$ are not in $\sto_e$. However the transvection $\phi$ mapping $x_1 \mapsto ax_1$ is equivalent in $\out(F_N)$ to the automorphism $\phi'$ sending $x_1$ to $x_1a$ and conjugating every other basis element by $a^{-1}$, so that $\Theta(\phi')=(a, \ldots, a; \ad_a^{-1})$. It follows that $L_1$ is mapped to the subgroup $J$ of  $M_{2N-5}(A)$ (using the notation of Section~\ref{s:slim_insert}). With some care one can then check that as in the Aut proof, $W_{N-2}$ acts on $M_{2N-5}(A)$ through $\Theta$ as a subgroup of the group $\sym(2N-4)$ defined in Section~\ref{s:slim_insert}. In particular, Proposition~\ref{p:augment} then tells us that the projection of $D$ to $W_{N-2}$ is trivial and $D < \stab^0(T)$. The result then follows by combining Theorem~\ref{t:all-the-D} with the above isomorphism and the observation that $\Theta(L_1)=J$. \end{proof}

\begin{proposition} \label{p:centralizer}
Let $N \geq 3$ and let $D < \aut(F_N)$ be a direct product of $2N-3$ nonabelian free groups in $\Aut(F_N)$. Then the centralizer of $D$ is either trivial or generated by a Nielsen automorphism.
\end{proposition}

\begin{proof}
Let $C=C(D)$ be the centralizer of $D$, and let $\overline{C}, \overline{D}$ be the respective images of $C$ and $D$ in $\Out(F_N)$. As in the proof of Theorem~\ref{t:B}, the group $\overline{D}$ is contained in the normalizer of a direct product of $2N-4$ nonabelian free groups, so by Theorem~A it fixes a unique collapsed rose with $N-2$ petals in the boundary of Outer space. By uniqueness $\overline{C}$ also fixes this rose, so $C$ acts on the associated Bass--Serre tree $T$. The group $D$ acts on $T$ with a unique fixed point (Proposition~\ref{p:fixed_point}), therefore $C$ also fixes this point. Hence $C$ is in the subgroup $H=M_{2N-4}(A) \rtimes W_{N-2}$ described in Proposition~\ref{p:stab-description-rose}. The second part of Proposition~\ref{p:augment} tells us that the projections of $C$ and $D$ to the $W_{N-2} < \sym(2N-3)$ factor are trivial, so $C < M_{2N-4}(A)$. We can therefore apply Part~2 of Theorem~\ref{t:all-the-D}:  if the centralizer of $D$ in $M_{2N-4}(A)$ is nontrivial it is generated by the element $\t$, and this is mapped to a Nielsen transformation $\tau$ under the isomorphism between the group $M_{2N-4}(A)$ and the point stabilizer in $T$ given in Proposition~\ref{p:stab-description-rose}.\end{proof}

\begin{corollary}[Theorem~\ref{t:nielsen}]
Let $N \geq 3$ and suppose $\G$ is a finite-index subgroup of $\Aut(F_N)$.  If $f \colon \G \to \Aut(F_N)$ is an injective map then every power of a Nielsen automorphism is mapped to a power of a Nielsen automorphism under $f$.
\end{corollary}

\begin{proof}
If $\phi$ is the power of a Nielsen automorphism then $\phi$ centralizes a direct product $D < \aut(F_N)$ of $2N-4$ nonabelian free groups (see the second case of Theorem~C). By taking finite-index subgroups of the factors, we can assume that $D < \Gamma$, so that $f(\phi)$ centralizes $f(D)$. Hence $f(\phi)$ is also a power of a Nielsen automorphism by Proposition~\ref{p:centralizer}.
\end{proof}

\section{Ascending chains of direct products}\label{s:maximal}

As well as taking direct products of free groups in $\aut(F_N)$ or $\out(F_N)$ with a maximal number of direct factors, one might also ask about maximality with respect to inclusion. For arbitrary groups, one must be very careful: if $G$ is a group with product rank $k$ and $\mathcal{D}$ is the poset of direct products of $k$ nonabelian free groups in $G$ (ordered by containment), it is not generally true that $\mathcal{D}$ contains maximal elements. One reason for this is the existence of \emph{locally free groups} that are not free.

Recall that a group $G$ is called \emph{locally free} if and only if every finitely generated subgroup of $G$ is free. If $G$ is countable, this is equivalent to the condition that $G$ is the direct limit of its free subgroups. The simplest example of a locally free group that is not free is $\mathbb{Q}$. The free product $\mathbb{Q} \ast \mathbb{Q}$ is also locally free and clearly contains nonabelian free groups. However, divisibility is not the only reason why a locally free group can fail to be free. We are grateful to Henry Wilton for directing us to the following example of Kurosh: take the one-relator group $G = \langle a,b,t \,|\, t[a,b]t^{-1}=a \rangle$, and let $H$ be the kernel of the map to $\mathbb{Z}$ given by $a,b \mapsto 0$, $t \mapsto 1$. The group $G$ is the fundamental group of the space $X_G$ obtained from a one-holed torus $T$  by gluing the boundary component of $T$ to a simple closed curve on $T$. The group $H$ is then the fundamental group of an infinite chain of surfaces (shown in Figure~\ref{f:Kurosh}).

\begin{figure}[ht]  \centering \def\svgwidth{400pt} 
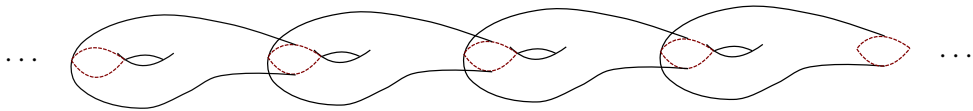 \caption{Kurosh's example of a locally free group that is not free is the fundamental group of the above infinite chain of tori.} \label{f:Kurosh} \end{figure}

The group $H$ has an infinite presentation with generating set $\{a_i,b_i: i \in \mathbb{Z}\}$ and relations  $[a_i,b_i]=a_{i+1}$ for all $i \in \mathbb{Z}$. The subgroup generated by $a_{-n}, b_{-n}, \ldots, a_n, b_n$ is freely generated by $a_{-n}$ and $b_{-n}, b_{-n+1}, \ldots, b_n$, which implies that $H$ is locally free. However, $H$ is not free as $H$ is not residually nilpotent: the relations let us write each $a_i$ as an arbitrarily long iterated commutator, so $a_i$ is trivial in every nilpotent quotient of $H$.

The group $G$ containing $H$ is well-behaved---it is hyperbolic and the fundamental group of a 3-manifold with boundary (this follows by constructing $X_G$ in $\mathbb{R}^3$ and thickening). Following the classification results in Section~\ref{s:applications}, we would like to rule out this behaviour in the mapping torus $M_\tau = F_2 \rtimes_\tau \mathbb{Z}$, where we will take $F_2=\langle a,b \rangle$ and $\tau$ the automorphism taking $a \mapsto ab$ and fixing $b$. In the spirit of the rest of the paper, we embed $M_\tau$ in $\aut(F_2)$ by identifying $F_2 \rtimes 1$ with the inner automorphisms

\begin{proposition}
A subgroup of $M_\tau$ is free if and only if it does not contain a subgroup isomorphic to $\mathbb{Z}^2$. Any free subgroup of $M_\tau$ is contained in a maximal one.
\end{proposition}

\begin{proof}
Note that the second statement follows from the first via Zorn's Lemma. If we have an ascending chain $H_1 < H_2 < H_3 < \cdots $ of free subgroups of $M_\tau$, then the union $H = \cup H_i$ does not contain a subgroup isomorphic to $\mathbb{Z}^2$ (as it would be contained in one of the $H_i$). Therefore $H$ is free, and we can apply Zorn's Lemma.

We are left with the trickier task of proving the first assertion. To do this, we look at the limiting tree $T$ of the automorphism $\tau$ in the boundary of Outer space (see \cite{CL}). This is a cyclic splitting with vertex stabilizers conjugate to $\stab(\tau)=\langle aba^{-1},b \rangle$ and edge stabilizers conjugate to $\langle b \rangle$. As $T$ is invariant under $[\tau]$, there is an action of $M_\tau$ on $T$ with vertex stabilizers conjugate to $\langle aba^{-1}, b \rangle\times\langle \tau\rangle$ and edge stabilizers conjugate to $\mathbb{Z}^2=\langle b, \tau\rangle$. At a vertex $v$, there are two $\stab(v)$--orbits of adjacent edges. If $\stab(v)=\langle aba^{-1} ,b \rangle$, these are the conjugacy classes of $\langle aba^{-1}, \tau \rangle$ and $\langle b, \tau \rangle$ in $\stab(v)$.  

Recall that a \emph{cylinder} in $T$ is a subtree $C_g$ that is fixed pointwise by a nontrivial element $g$ of $M_\tau$. In $F_2$ edge stabilizers are malnormal, so if a cylinder $C_g$ contains more than one edge then $g \not \in F_2$. Hence $g=\ad_x \tau^k$ for some $x \in F_2$ and $k \neq 0$. The subgroup of $F_2$ fixed by $g$ is nonabelian, as $g$ commutes with any inner automorphism fixing an edge in its cylinder. Hence $g$ is similar to $\tau^k$ by Collins--Turner (Proposition~\ref{p:CT}). Hence the cylinder is a star of radius one (if $\tau$ fixes an edge then the edge is adjacent to the vertex with stabilizer $\stab(\tau)$). This shows that every cylinder in $T$ is either a point, a single edge, or a star of radius one. 

Let $H$ be a subgroup of $M_\tau$ that does not contain $\mathbb{Z}^2$. Then for every vertex $v \in T$, the stabilizer $H_v$ is free. As above we may assume that $\stab(v) = \langle aba^{-1}, b \rangle \times \langle \tau \rangle$. The `exceptional' case is where $H_v \cap \langle \tau\rangle$ is nontrivial. Then $H_v < \langle \tau\rangle$  as $H$ contains no $\mathbb{Z}^2$ subgroups. Then the $H$--stabilizer of every edge adjacent to $v$ is also equal to $H_v$. The `generic' case is where $H_v \cap \langle \tau \rangle$ is trivial. Then $H_v$ embeds into $\langle aba^{-1},b \rangle$ via the projection to this factor, and under this projection each adjacent edge stabilizer is contained in a $\stab(v)$--conjugate of $\langle aba^{-1} \rangle$ or $\langle b \rangle$. Hence $H_v$ splits relative to its adjacent edge groups via the free splitting $S_v:=\langle aba^{-1}\rangle \ast\langle b\rangle$.

One can therefore blow up each generic vertex group via the splitting $S_v$. In the new $H$--tree $T'$, the nontrivial edge stabilizers are equal to their adjacent vertex stabilizers, so the tree is a union of cylinders with identical cyclic edge and vertex groups, necessarily separated by edges with trivial stabilizers. As all new edges in the blow-up have trivial stabilizers, each cylinder in $T'$ is unchanged from $T$ and is either a point, a single edge, or a star of radius one. It follows that any setwise stabilizer of a cylinder fixes a point in that cylinder, and therefore fixes the cylinder pointwise (as the edge and vertex groups in cylinders are all identical). This means we can collapse each cylinder to a point to give a tree $T''$ on which $H$ acts with trivial edge stabilizers and vertex stabilizers that are either $\mathbb{Z}$ or trivial. Hence $H$ is free.
  \end{proof}

\begin{corollary}
Let $N \geq3$ and let $\mathcal{D}$ be the family of subgroups of either $\aut(F_N)$ or $\out(F_N)$ that are direct products of $2N-3$ or $2N-4$ nonabelian free groups, respectively. Then every $D \in \mathcal{D}$ is contained in a maximal element (with respect to inclusion).
\end{corollary}

\begin{proof}
By either Theorem~\ref{t:direct-products-aut} or Theorem~\ref{t:direct-products-out}, this is clear unless one factor of $D$ is contained in a subgroup of $\aut(F_N)$ or $\out(F_N)$ isomorphic to $M_\tau$. However, from the above work we know that every free subgroup of $M_\tau$ is contained in a maximal one, from which the result follows.
\end{proof}

\bibliography{direct-products-bib}

\begin{flushleft} 
Mathematical Institute\\
University of Oxford\\
Oxford OX2 6GG\\
\emph{e-mails: }\texttt{bridson@maths.ox.ac.uk, wade@maths.ox.ac.uk} 
\end{flushleft}

\end{document}